\documentclass[reqno]{amsart}
\usepackage{amssymb,amsmath,amsfonts}
\usepackage{hyperref} 
\newtheorem{teor}{Theorem}[section]
\newtheorem{lema}[teor]{Lemma}

\newtheorem{coro}[teor]{Corollary}
\theoremstyle{definition}
\newtheorem{defi}[teor]{Definition}

\newtheorem{hipo}[teor]{Hypothesis}

\newtheorem{nota}[teor]{Remark}

\numberwithin{equation}{section}
\newcommand{\F}{\mathbb F}
\newcommand{\R}{\mathbb R}

\newcommand{\K}{\mathbb{K}}

\newcommand{\C}{\mathbb{C}}
\newcommand{\N}{\mathbb{N}}

\newcommand{\M}{\mathbb{M}}
\newcommand{\s}{\mathbb{S}}
\newcommand{\mC}{\mathcal{C}}

\newcommand{\mB}{\mathcal{B}}

\newcommand{\mM}{\mathcal{M}}

\newcommand{\mK}{\mathcal{K}}

\newcommand{\ep}{\varepsilon}
\newcommand{\J}{J^{-1}\!}

\newcommand{\mI}{\mathcal{I}}
\newcommand{\mJ}{\mathcal{J}}
\newcommand{\mL}{\mathcal{L}}
\newcommand{\mQ}{\mathcal{Q}}

\newcommand{\W}{\Omega}
\newcommand{\w}{\omega}

\newcommand{\G}{\Gamma}
\newcommand{\Lk}{\mathcal{L}_{\K}}

\newcommand{\bv}{\mathbf v}

\newcommand{\bw}{\mathbf w}
\newcommand{\bx}{\mathbf x}
\newcommand{\by}{\mathbf y}
\newcommand{\bz}{\mathbf z}
\newcommand{\bu}{\mathbf u}

\newcommand{\bcero}{\mathbf 0}
\newcommand{\wit}{\widetilde}
\newcommand{\ws}{\w{\cdot}s}
\newcommand{\wt}{\w{\cdot}t}
\newcommand{\wD}{\w{\cdot}d}

\newcommand{\n}[1]{\left\|#1\right\|}
\newcommand{\Frac}[2]{\displaystyle\frac{#1}{#2}}

\newcommand{\lsm}{\left[\begin{smallmatrix}}
\newcommand{\rsm}{\end{smallmatrix}\right]}

\DeclareMathOperator{\cls}{cls}

\DeclareMathOperator{\Ima}{Im}
\DeclareMathOperator{\Rea}{Re}

\DeclareMathOperator{\Supp}{Supp}

\begin{document}
\title[Non-Atkinson perturbations of
linear Hamiltonian systems]
{Non-Atkinson perturbations of nonautonomous
linear Hamiltonian systems:
exponential dichotomy and nonoscillation}
\author[C. N\'{u}\~{n}ez]{Carmen N\'{u}\~{n}ez}
\author[R. Obaya]{Rafael Obaya}
\address{Departamento de Matem\'{a}tica Aplicada, EII,
Universidad de Valladolid, Paseo del Cauce 59, 47011 Valladolid, Spain.}
\email[Carmen N\'{u}\~{n}ez]{carnun@wmatem.eis.uva.es}
\email[Rafael Obaya]{rafoba@wmatem.eis.uva.es}
\thanks{Partly supported by MINECO/FEDER (Spain)
under project MTM2015-66330-P and by European Commission under project
H2020-MSCA-ITN-2014.}
\date{}
\begin{abstract}
We analyze the presence of exponential dichotomy (ED) and of global
existence of Weyl functions $M^\pm$
for one-parametric families of finite-dimensional
nonautonomous linear Hamiltonian
systems defined along the orbits of a compact metric space,
which are perturbed from an initial one in a direction which does not satisfy
the classical Atkinson condition: either they do not have ED for any
value of the parameter; or they have it for at least all the nonreal values,
in which case the Weyl functions exist and are Herglotz.
When the parameter varies in the real line,
and if the unperturbed family satisfies the properties
of exponential dichotomy and global existence
of $M^+$, then these two properties persist
in a neighborhood of 0 which agrees either with the whole real line or with
an open negative half-line; and in this last case, the ED fails at the
right end value. The properties of ED and of global existence of $M^+$ are
fundamental to guarantee the solvability of classical
minimization problems given by linear-quadratic control processes.
\end{abstract}
\maketitle
\section{Introduction}
The theory of exponential dichotomy has played a central role in the study of finite
and infinite dimensional dynamical systems, including those arising in the
analysis of nonautonomous differential
equations. In the linear case, the occurrence of exponential dichotomy
is directly connected with the invertibility of the associated operators.
And, in the nonlinear case, the robustness of the exponential dichotomy of
the linearized flows converts this property in an essential tool
to analyze the behavior of the solutions.
\par
In particular, the exponential dichotomy is also fundamental
in the description of invariant manifolds,
perturbation problems, bifurcation patterns,
homoclinic trajectories and spectral theory, among many other questions. The
works of Coppel~\cite{copp, copp1}, Massera and Schaeffer \cite{masc},
Hale~\cite{hale}, Sacker and Sell
\cite{sase2, sase3, sase4, sase5, sase7, sase8}, Sell \cite{sell3},
Chow and Hale \cite{chha}, Palmer \cite{palm3,palm4}, Johnson~\cite{john20},
Vanderbauwhede and van Gils \cite{vava}, Vanderbauwhede \cite{vand},
Henry \cite{henr}, Johnson and Yi~\cite{joyi}, Chow and
Leiva \cite{chle2, chle3}, Shen and Yi \cite{shyi4},
Chicone and Latushkin \cite{chla}, Pliss and Sell
\cite{plse}, and Johnson {\em et al.} \cite{jonnf}
(which compose an incomplete list) provide an exhaustive
analysis of all these topics.
\par
In the classical field of finite-dimensional
linear Hamiltonian differential equations
with periodic time-dependent coefficients, the existence and robustness
of the exponential dichotomy is directly related to
the regions of instability and total instability studied by
Gel'fand and Lidski\u{\i} \cite{geli}
and Yakubovich \cite{yaku3, yaku4}. In the more general setting
of nonautonomous linear Hamiltonian systems with bounded and uniformly continuous
coefficients, these questions are extensively analyzed in the book Johnson {\em et al.}
\cite{jonnf}, which presents a unified version of many previous works due to the
authors of the book and to many other researchers.
In particular, in this book, the applicability of the exponential dichotomy results
to the study of nonautonomous linear-quadratic control
processes is extensively analyzed. We will now explain briefly a point of this
analysis, which is central in order to understand the scope of the present
paper.
\par
Let us consider the control problem
\begin{equation}\label{1.control}
 \bx' = A_0(t)\,\bx+ B_0(t)\,\bu\,,
\end{equation}
where $\bx\in\R^n$ and $\bu\in\R^m$, together with the quadratic form
(supply rate)
\[
 \mQ(t,\bx,\bu):=\frac{1}{2}\left(\langle\bx, G_0(t)\,\bx\rangle
 +2\langle\bx,g_0(t)\,\bu\rangle+\langle\bu, R_0(t)\,\bu\rangle\right).
\]
The functions $A_0$, $B_0$, $G_0$, $g_0$, and $R_0$
are assumed to be bounded and uniformly
continuous functions on $\R$, with values in the sets of real matrices
of the appropriate dimensions. In addition, $G$ and
$R$ are symmetric, and $R(t)\ge\rho I_m$ for a common $\rho>0$ and all $t\in\R$.
We also fix $\bx_0\in\R^n$ and introduce the quadratic functional
\[
 \mJ_{\bx_0}(\bx,\bu):= \int_0^{\infty} \mQ(t,\bx(t),\bu(t))\,dt
\]
evaluated on the so-called {\em admissible pairs\/}
$(\bx,\bu)\colon[0,\infty)\to\R^n\times\R^m$; i.e. those for which
$\bu$ belongs to $L^2((0,\infty),\R^m)$ and the solution $\bx(t)$ of
\eqref{1.control} for this control with $\bx(0)=\bx_0$
belongs to $L^2((0,\infty),\R^n)$.
The problem to consider is that of minimizing $\mJ_{\bx_0}$ relative to the
set of admissible pairs.
\par
By means of a standard construction (the so called hull or Bebutov
construction, which we will summarize in Section \ref{sec2}),
this problem can be included in a family, given by the control problems
\begin{equation}\label{7.controlAB2}
 \bx'=A(\wt)\,\bx+B(\wt)\,\bu
\end{equation}
and by the functionals
\[
\begin{split}
 \mQ_\w(t,\bx,\bu)&:=\Frac{1}{2}\left(\langle\bx, G(\wt)\,\bx\rangle
 +2\langle\bx,g(\wt)\,\bu\rangle+\langle\bu, R(\wt)\,\bu\rangle\right),\\
 \mJ_{\bx_0,\w}(\bx,\bu)&:=
 \int_0^{\infty} \mQ_\w(t,\bx(t),\bu(t))\,dt
\end{split}
\]
for $\w\in\W$ and $\bx_0\in\R^n$.
Here, $\W$ is a compact metric space admitting a continuous flow,
$\wt$ is the orbit of a point $\w\in\W$, $A$, $B$, $G$, $g$, and $R$
are bounded and uniformly continuous matrix-valued functions on $\W$,
$G$ and $R$ are symmetric, and $R>0$. It is important to point out that
$\W$ is minimal in the case of recurrence of the initial coefficients,
which includes the autonomous, periodic, quasi-periodic, almost-periodic and
almost-automorphic cases.
\par
The Pontryagin Maximum Principle relates the problem of
minimizing $\mJ_{\bx_0,\w}$ to the properties of the
family of linear Hamiltonian systems
\begin{equation}\label{1.hamil}
 \bz'=H(\wt)\,\bz\,,\qquad \w\in\W\,,
\end{equation}
where $\bz=\lsm\bx\\\by\rsm$ for $\bx,\by\in\R^n$ and
\[
 H=\left[\begin{array}{cc} A-B\,R^{-1}g^T&\;
                        B\,R^{-1}B^T\\
                   G-g\,R^{-1}g^T &\; -A^T\!+g\,R^{-1}B^T
                   \end{array}\right].
\]
More precisely, under a certain uniform stabilization condition,
it turns out that the minimization problem for each one of the functionals
$\mJ_{\w,\bx_0}$ is solvable if the family \eqref{1.hamil} admits
exponential dichotomy and if, in addition, for every $\w\in\W$,
the Lagrange plane $l^+(\w)$
composed by those initial data $\bz_0$ giving rise to a solutions
bounded at $+\infty$ admits a basis whose vectors compose the columns
of a matrix $\lsm I_n \\M^+(\w)\rsm$.
In other words, if the family admits
exponential dichotomy (or the {\em frequency condition\/} holds)
and the Weyl function $M^+$ globally exists
(or the {\em nonoscillation condition\/} is satisfied).
(For further purposes we
point out that the Weyl function $M^-$ is associated in the
analogous way to the Lagrange plane
$l^-(\w)$ composed by the initial data $\bz_0$ of the solutions
bounded at $-\infty$.)
In addition, if this is the case, the unique
minimizing pair $(\wit\bx(t),\wit\bu(t))$ for
$\mJ_{\w,\bx_0}$ is determined from the solution $\lsm\wit\bx(t)\\\wit\by(t)\rsm$
of \eqref{1.hamil} with initial data $\lsm\bx_0\\M^+(\w)\,\bx_0\rsm$
by means of the feedback rule
$\wit\bu(t)=R^{-1}(\wt)\,B^T(\wt)\,\wit\by(t)- R^{-1}(\wt)\,g^T(\wt)\,\wit\bx(t)$.
And, as a matter of fact, both situations (solvability and \lq\lq frequency
plus nonoscillation conditions\rq\rq) are equivalent in many dynamical
situations, as in the case of minimality of $\W$. This result,
first published in Fabbri {\em et al.} \cite{fajn4} and \cite{fjnn2} (and
which is extremely detailed in Chapter 7 of \cite{jonnf}),
constitutes a nonautonomous version of the Yakubovich Frequency
Theorem for the periodic case, which appears in \cite{yaku3,yaku4}.
\par
The historical and practical importance of the above
result justifies the interest of this paper, whose central goal is
to analyze the presence and preservation of the exponential dichotomy
and the nonoscillation condition in parametric families of linear Hamiltonian
systems.
\par
In what follows, we explain simultaneously the structure of the paper and its main
achievements. Section \ref{sec2} summarizes some basic notions on topological
dynamics, and explains with some detail the concepts of exponential dichotomy, uniform
weak disconjugacy, and rotation number, which are fundamental in the statements
and proofs of the main results. We also summarize some of the many results
relating these concepts.
\par From now on, $\W$ is a compact metric space with a continuous flow,
and we represent by $\{\wt\,|\;t\in\R\}$ the orbit of the element $\w\in\W$.
In addition, $H_1,H_2,H_3$
and $\Delta$ are continuous $n\times n$ matrix-valued functions
on $\W$, and $H_2,H_3$ and $\Delta$ take symmetric values.
\par
In Section \ref{sec3} we consider the
families of linear Hamiltonian systems
\begin{equation}\label{1.hamilper}
\bz'=H^\lambda(\wt)\,\bz\,,\quad \text{where }\;
 H^\lambda:=\left[\begin{array}{cc} H_1&
 H_3-\lambda\Delta\\
 H_2&-H_1^T\end{array}\right]
\end{equation}
for $\w\in\W$. The parameter $\lambda$ varies in $\C$. If the
matrix-valued function $\G:=\lsm 0_n&0_n\\0_n&\Delta\rsm$ satisfies
the so-called Atkinson definiteness condition (see Atkinson~\cite{atki}),
then the systems
\eqref{1.hamilper}$^\lambda$ satisfy the frequency condition
and admit both Weyl functions
at least for $\lambda\in\C-\R$. This is an already classical
result due to Johnson \cite{john1}. Here, we analyze the problem
for $\Delta>0$ without imposing the Atkinson hypothesis, and prove that two
dynamical possibilities arise: either the families
\eqref{1.hamilper}$^{\lambda}$ have exponential dichotomy
for (at least) all $\lambda\in\C-\R$,
in which case the Weyl functions $M^\pm(\w,\lambda)$ globally exist
and are Herglotz functions; or
the family~\eqref{1.hamilper}$^\lambda$
does not have exponential dichotomy
for any $\lambda\in\C$, which turns out to be
equivalent to the existence of a point $\w\in\W$ and
a nonzero bounded function of the form
$\bz(t,\w)=\lsm\bz_1(t,\w)\\\bcero\rsm$
which solves the system~\eqref{1.hamilper}$^\lambda$
corresponding to $\w$ for all $\lambda\in\C$. These results (excepting
the existence of Weyl functions) are easily
transferable to the families
\begin{equation}\label{1.hamilper2}
\bz'=\wit H^\lambda(\wt)\,\bz\,,\quad \text{where }\;
 \wit H^\lambda:=\left[\begin{array}{cc} H_1&
 H_3\\
 H_2-\lambda\Delta&-H_1^T\end{array}\right].
\end{equation}
In this case, the second dynamical situation is equivalent to
the existence of a point $\w\in\W$ and
a nonzero bounded function of the form
$\bz(t,\w)=\lsm\bcero\\\bz_2(t,\w)\\\rsm$ which solves
the system~\eqref{1.hamilper2}$^\lambda$
corresponding to $\w$ for all $\lambda\in\C$. In particular,
all the systems corresponding to $\w$ are {\em abnormal\/} systems.
This type of systems have been extensively studied during the last
decades: see e.g.~Reid~\cite{reid3,reid8},
Kratz~\cite{kratz}, \v{S}epitka and \v{S}imon Hilscher \cite{sesi1,sesi2,sesi3}
Fabbri~{\em et al.}~\cite{fnno},
Johnson {\em et al.}~\cite{jnno2}, and references
therein.
\par
In Section \ref{sec4} we go further in the analysis of the families
\eqref{1.hamilper2} with $H_3\ge 0$ and $\Delta>0$.
More precisely, we assume that \eqref{1.hamilper2}$^0$
has exponential dichotomy (ED) and satisfies the nonoscillation condition
(NC), and define
\[
 \mI:=\{\lambda\in\R\,|\;\eqref{1.hamilper2}^\lambda \text{ has
 ED and satisfies NC\/}\}\,.
\]
Under the assumption of existence of
an ergodic measure on $\W$ with full topological support
(which holds at least in the case of the minimality of $\W$),
we prove among other properties
that $\mI$ is either the whole line or an open negative half-line; and that,
in addition, if $\mI=(-\infty,\lambda^*)$ for a real $\lambda^*$, then the family
\eqref{1.hamilper2}$^{\lambda^*}$ does not have exponential dichotomy.
This result improves and extends a previous theorem of Johnson
{\em et al.}~\cite{jnuo}. In its proof a fundamental
role is played by the occurrence of uniform weak disconjugacy and by the properties
of the rotation number: both properties are fundamental
to determine the presence of exponential dichotomy, in different
settings. The reader is referred to
Johnson {\em et al.}~\cite{jono2}, Fabbri {\em et al.} \cite{fajn5,fjnn1},
Johnson {\em et al.}~\cite{jnuo, jnno} and Chapter 5 of \cite{jonnf} for an
in-depth analysis of the uniform weak disconjugacy property, and to
Johnson~\cite{john1}, Novo {\em et al.}~\cite{nono},
Fabbri {\em et al.} \cite{fajn1,fajn2} and Chapter 2 of \cite{jonnf} for
the definition and main properties of the rotation number.
The result concerning the shape and properties of $\mI$ is finally extended to the case
in which the base flow is distal.
\par
This paper is dedicated to the memory of George Sell. Among his large
number of achievements, the development of the theory of
exponential dichotomies for nonautonomous
dynamical systems given by
 skew-product flows on vector
bundles with compact base, is more than fundamental in the
work of the authors of this paper.
\section{Preliminaries}\label{sec2}
All the contents of this preliminary section can be found in
Johnson {\em et al.}~\cite{jonnf}, where the reader will also find
a quite exhaustive list of references for the origin of the results
that we summarize here.
\par
Let us begin by establishing some notation. As usual, $\R$ and $\C$
represent the real line and the complex plane. If $\lambda\in\C$,
$\Rea\lambda$ and $\Ima\lambda$ are respectively its real and imaginary parts.
\par
Now let $\K$
represent $\R$ or $\C$. The set $\M_{d\times m}(\K)$ is the set
of $d\times m$ matrices with entries in $\mK$.
As usual, $\K^d:=\M_{d\times 1}(\K)$, and
$A^T$ is the transpose of the matrix $A$. The subset
$\s_d(\K)\subset\M_{d\times d}(\K)$ is composed by the symmetric matrices.
If $M\in\s_d(\R)$ is symmetric, the expressions
$M>0$, $M<0$, $M\ge 0$, and $M\le 0$
mean that it is positive definite, positive semidefinite,
negative definite, and negative semidefinite. If $\W$ is a topological
space and $M\colon\W\to\s_d(\K)$ is a map, $M>0$ means that $M(\w)>0$ for all
the elements $\w\in\W$, and $M<0$, $M\ge 0$, and $M\le 0$ have the
analogous meaning. It is also obvious what $M_1>M_2$, $M_1\ge M_2$,
$M_1<M_2$, and $M_1\le M_2$ mean. We represent by $I_d$ and $0_d$ the
identity and zero $d\times d$ matrices, and by $\bcero$ the null
vector of $\K^d$ for all $d$.
If $\bz\in\K^d$, its Euclidean norm is $\n{\bz}$, and if
$A\in \M_{d\times m}(\K)$, then $\n{A}$ is the
associated operator norm.
\par
A (real or complex) {\em Lagrange plane\/} is an $n$-dimensional
(real or complex) linear space such that
$\bz^TJ\,\bw=0$ for any pair of elements $\bz$ and $\bw$,
where $J=\lsm 0_n&-I_n\\I_n&0_n\rsm$.
A Lagrange plane $l$ {\em is represented by\/} $\lsm L_1\\ L_2\rsm$
if the column vectors of the matrix form a basis of the $n$-dimensional linear
space $l$. Hence, it can be also represented
by $\lsm I_n\\M\rsm$ if and only if $\det L_1\ne 0$, in which case the matrix
$M=L_2L_1^{-1}$ is symmetric.
\par
Now we will recall some basic concepts and properties of
topological dynamics and measure theory.
Let $\W$ be a complete metric space. A ({\em real and continuous})
{\em global flow\/} on $\W$ is a continuous map
$\sigma\colon\R\times\W\to\W,\; (t,\w)\mapsto\sigma(t,\w)$
such that
$\sigma_0=\text{Id}$ and $\sigma_{s+t}=\sigma_t\circ\sigma_s$ for each $s,t\in\R$,
where $\sigma_t(\w)=\sigma(t,\w)$.
The flow is {\em local\/} if the map $\sigma$
is defined, continuous, and satisfies the previous properties on an open subset
of $\R\times\W$ containing $\{0\}\times\W$.
\par
Let $(\W,\sigma)$ be a global flow.
The $\sigma$-{\em orbit\/} of a point $\w\in\W$
is the set $\{\sigma_t(\w)\,|\;t\in\R\}$.
Restricting the time to $t\ge 0$ or
$t\le 0$ provides the definition of {\em forward\/}
or {\em backward\/} $\sigma$-semiorbit.
A subset $\W_1\subset \W$ is {\em $\sigma$-invariant\/}
if $\sigma_t(\W_1)=\W_1$ for every $t\in\R$.
A $\sigma$-invariant subset $\W_1\subset\W$ is {\em minimal\/} if it is compact
and does not contain properly any other compact $\sigma$-invariant set;
or, equivalently, if each one of the two semiorbits of anyone of
its elements is dense in it. The continuous flow $(\W,\sigma)$ is
{\em minimal\/} if $\W$ itself is minimal. And the flow is {\em distal}
if,  whenever $\w_1\ne\w_2$, there exists $d(\w_1,\w_2)>0$
such that the distance
between $\w_1\cdot t$ and $\w_2\cdot t$ is greater that $d(\w_1,\w_2)$ for
all $t\in\R$.
\par
If the set $\{\sigma_t(\w)\,|\;t\ge 0\}$ is relatively compact,
the {\em omega limit set\/} of $\w_0$ is given
by those points $\w\in\W$ such that $\w=\lim_{m\to \infty}\sigma(t_m,\w_0)$
for some sequence $(t_m)\uparrow \infty$. This set is
nonempty, compact, connected and $\sigma$-invariant.
The definition and properties of the {\em alpha limit set\/} of $\w_0$
are analogous, working now with sequences $(t_m)\downarrow-\infty$.
\par
Let $m$ be a normalized Borel measure on $\W$; i.e.~a finite
regular measure defined on the Borel subsets of $\W$ and with
$m(\W)=1$. The measure $m$ is {\em $\sigma$-invariant\/}
if $m(\sigma_t(\W_1))=m(\W_1)$ for every Borel subset
$\W_1\subset\W$ and every $t\in\R$. If, in addition,
$m(\W_1)=0$ or $m(\W_1)=1$ for every $\sigma$-invariant subset
$\W_1\subset\W$, then the measure $m$ is {\em $\sigma$-ergodic\/}.
If $\W$ is a compact, a real continuous flow $(\W,\sigma)$
admits at least an ergodic measure.
And the {\em topological support\/} of $m$, $\Supp m$, is the complement
of the largest open set $O\subset\W$ for which $m(O)=0$. In the case that $\W$
is minimal, then it agrees with the topological support of any $\sigma$-ergodic
measure.
\par
In the rest of the paper, $(\W,\sigma)$ will be a real continuous
global flow on a compact metric space, and we will
denote $\wt=\sigma(t,\w)$. Recall that we represent by $\K$ either $\R$ or $\C$. Our
starting point is the family of linear Hamiltonian systems
\begin{equation}\label{2.hamil}
 \bz' = H(\wt)\,\bz\,,\quad\w\in\W\,,
\end{equation}
where $H\colon\W\to\mathfrak{sp}(n,\K)$ is continuous.
Here, $\mathfrak{sp}(n,\K)$ is the
Lie algebra of the infinitesimally symplectic matrices,
\[
 \mathfrak{sp}(n,\K):=\{H\in\M_{2n\times 2n}(\K)\,|\;H^T\!J+JH=0_{2n}\}\,,
\]
so that, since $J=\lsm 0_n&-I_n\\I_n&0_n\rsm$, $H$ takes the form
\[
 H(\w)=\left[\begin{array}{cc} H_1(\w)& H_3(\w)\\
 H_2(\w)&-H_1^T(\w)\end{array}\right],
\]
with $H^T_2=H_2$ and $H_3^T=H_3$.
Let $U(t,\w)$ denote the fundamental matrix solution of the system
\eqref{2.hamil} for $\w\in\W$ with $U(0,\w)=I_{2n}$.
The family \eqref{2.hamil} induces
a real continuous global flow
on the linear bundle $\W\times\K^{2n}$, given by
\begin{equation}\label{2.deftau}
 \tau_\K\colon\R\times\W\times\K^{2n}\to\W\times\K^{2n}\,,\quad(t,\w,\bz)\mapsto
 (\wt,U(t,\w)\,\bz)\,.
\end{equation}
This flow is called of {\em skew-product type} since its first
component agrees with the base flow, and {\em linear} since the second component
is a linear map for each $\w\in\W$.
\par
Frequently, a family of this type comes from a single nonautonomous Hamiltonian
system $\bz'=H_0(t)\,\bz$ by means of the well known Bebutov construction:
if $H_0$ is bounded and uniformly continuous on $\R$, then its {\em hull\/}
$\W$, which is defined by $\W:=\cls \{H_t\,|\;t\in\R\}$ (where
$H_t(s) =H_0(t+s)$ and the closure is taken in the compact-open topology),
is a compact metric space; and the time-translation defines
a continuous flow $\sigma$ on it. The base space $\W$
can hence be understood as the space in which the nonautonomous law varies with
respect to time. Under additional recurrence properties on $H_0$, the
base flow is minimal. This is the case if $H_0$ is almost periodic or
almost automorphic. Weaker conditions on
$H_0$ may provide a non minimal hull, which can contain
different minimal subsets. In some of these cases the solutions
of the different linear Hamiltonian systems of the family may
show a significatively different qualitative behavior.
\par
However, we will not assume that the family \eqref{2.hamil} comes from a
single equation by means of the Bebutov construction, which makes our analysis
more general.
\par
In the rest of this section we recall some basic concepts and some
associated properties related to families of the form~\eqref{2.hamil}.
The analysis contained in this paper either concerns these properties
(this is the case of the {\em exponential dichotomy\/},
{\em nonoscillation condition}, and {\em uniform weak disconjugacy})
or requires them as tools for the proofs (as in the case
of the {\em rotation number}).
\begin{defi}\label{2.defED}
The family~\eqref{2.hamil}
has {\em exponential dichotomy\/} (or {\em ED\/} for short)
{\em over\/} $\W$ if there exist
constants $\eta\ge 1$ and $\beta>0$ and a splitting $\W\times
\K^{2n}=L^+\oplus L^-$ of the bundle into the Whitney sum of
two closed subbundles such that
\begin{itemize}
 \item[-] $L^+$ and $L^-$ are invariant
 under the flow $\tau_\K$ given by \eqref{2.deftau} on $\W\times\K^{2n}$;
 that is, if $(\w,\bz)$ belongs to $L^+$ (or to $L^-$), so does
 $(\wt,U(t,\w)\,\bz)$ for all $t\in\R$.
 \item[-] $\n{U(t,\w)\,\bz} \le \eta\,e^{-\beta t}\n{\bz}\quad$ for
   every $t\ge 0$ and $(\w,\bz)\in L^+$.
 \item[-] $\n{U(t,\w)\,\bz} \le \eta\,e^{\beta t}\n{\bz}\quad\;\;\,$
   for every $t\le 0$ and $(\w,\bz)\in L^-$.
\end{itemize}
\end{defi}
We will omit the words \lq\lq over $\W$" when the family~\eqref{2.hamil}
has ED, since no confusion arises.
Let us summarize in the next list of remarks some
well-known fundamental properties
satisfied by a family of linear Hamiltonian systems which has ED. Detailed
proofs and the names of the authors of the results can be
found in Chapter 1 of \cite{jonnf}.
\begin{nota}\label{2.notasED}
(a)~The ED is unique (in the sense that so are
the subbundles $L^+$ and $L^-$), and it precludes the existence of
globally bounded solutions for any of the systems of the family
\eqref{2.hamil}. These assertions are also true
when the family of linear systems is not of Hamiltonian type:
see e.g.~Section 1.4.1 of \cite{jonnf}.
\par
(b)~As a matter of fact,
the family~\eqref{2.hamil} has ED if
and only if no one of its systems has a nonzero bounded solution. And the
ED of the whole family is equivalent to the ED over $\R$ of each one of its systems.
\par
(c)~The sections
\begin{equation}\label{2.defl}
 l^\pm(\w):=\{\bz\in\K^{2n}\,|\;(\w,\bz)\in L^\pm\}
\end{equation}
are real Lagrange planes. In addition,
\begin{equation}\label{2.deflpm}
\begin{split}
 l^\pm(\w)&=\{\bz\in\K^{2n}\,|\;\lim_{t\to\pm\infty}\n{U(t,\w)\,\bz}=\bcero\}\\
 &=\{\bz\in\K^{2n}\,|\;\sup_{\pm t\in[0,\infty)}\n{U(t,\w)\,\bz}<\infty\}\,;
\end{split}
\end{equation}
and \vspace{-.2cm}
\begin{equation}\label{2.noacotada}
 \lim_{t\to\pm\infty}\n{U(t,\w)\,\bz}=
 \infty \quad\text{if $\bz\notin l^\pm(\w)$}\,.
\end{equation}
\par
(d)~Assume that for all $\w\in\W$, the Lagrange plane
$l^+(\w)$ can be represented by the matrix $\lsm I_n\\M^+(\w)\rsm$.
Or, equivalently, that for all $\w\in\W$, the Lagrange plane
$l^+(\w)$ can be represented by a matrix $\lsm L_1^+(\w)\\L_2^+(\w)\rsm$
with $\det L_1^+(\w)\ne 0$ (so that $M^+(\w)=L_2(\w)\,L_1^{-1}(\w)$).
In this case
$M^+\colon\W\to\s_n(\K)$ is a continuous matrix-valued
function, and it is known as one of the {\em Weyl functions}
for \eqref{2.hamil}. In this situation, we say that
the Weyl function $M^+$~{\em globally exists}.
In addition, for all $\w\in\W$
the function $t\mapsto M^+(\wt)$ is a solution of the Riccati equation
associated to \eqref{2.hamil}, namely
\begin{equation}\label{2.riccati}
  M'=-MH_3(\wt)M-MH_1(\wt)-H_1^T(\wt)M+H_2(\wt)\,.
\end{equation}
We say that $M^+$ is a {\em globally defined solution along the flow of\/}
\eqref{2.riccati}. The other Weyl function is $M^-$, associated
to the subbundle $L^-$, and it satisfies the same properties (if it exists).
\par
(e)~Now we do not assume the presence of ED.
Let $M(t,\w,M_0)$ represent the solution of the equation
\eqref{2.riccati} corresponding to $\w$ which satisfies $M(0,\w,M_0)=M_0$.
Then the map $(t,\w,M_0)\mapsto M(t,\w,M_0)$ defines a continuous
skew-product flow
on $\W\times\s_n(\R)$, which is in general local, since the solutions
may not be globally defined.  In particular,
$M(t+s,\w,M_0)=M(t,\ws,M(s,\w,M_0))$
whenever all the elements in the right-hand term are defined.
\end{nota}
\begin{defi}\label{2.defNC}
Suppose that the family \eqref{2.hamil} has ED. Then it
{\em satisfies the non\-oscillation condition\/}
(or {\em NC\/} for short) if the Weyl function $M^+$
globally exists.
\end{defi}
The next result is a consequence of the Sacker and Sell perturbation
theorem (Theorem 6 of \cite{sase5}), adapted to the particular setting
of families of linear Hamiltonian systems. It summarizes part of the information provided by Theorems 1.92 and 1.95 of \cite{jonnf}.
In particular, it proves that ED, NC, and the global
existence of $M^-$ are robust properties, in
the sense that each one of them persists under small perturbations of the
matrix $H$ of the family~\eqref{2.hamil}. The space of continuous $\mathfrak{sp}(n,\K)$-valued functions on $\W$
is provided with the topology of the uniform convergence.
And the space $\Lk$ of the (real or complex) Lagrange planes
is endowed with the topology as submanifold of the
Grassmannian manifold of $n$-dimensional linear subspaces of $\K^{2n}$
(see Sections 1.2.2 and 1.2.3 of \cite{jonnf} for further details).
\begin{teor}\label{2.teorSSPT}
Suppose that the family \eqref{2.hamil}has
exponential dichotomy over $\W$. Then there exists $\eta>0$ such that
if $\mB_\eta\subset C(\W,\mathfrak{sp}(n,\K))$ is the open
ball centered at $0_{2n}$ of radius $\eta$, then
the family $\bz'=(H(\wt)+K(\wt))\,\bz$
has exponential dichotomy over $\W$ for all $K\in\mB_\eta$.
\par
Let us represent the corresponding Lagrange planes
for $K\in\mB_\eta$ by $l^\pm_K(\w)$, and the Weyl functions (if they exist)
by $M^\pm_K(\w)$. Then,
\begin{itemize}
\item[\rm(i)] the maps $l^\pm\colon\W\times\mB_\eta\to\Lk\,,\;
(\w,K)\mapsto l^\pm_K(\w)$ are continuous.
\item[\rm(ii)] Suppose further that the function
$M^+_{0_{2n}}$, associated to the unperturbed family \eqref{2.hamil},
globally exists. Then $\eta>0$ can be chosen in such a way that
$M^+_K$ globally exists for all $K\in\mB_\eta$. In addition,
the maps $\W\times\mB_\eta\to \s_n(\K),\;(\w,K)\mapsto M^+_K(\w)$ and
$M^+\colon\mB_\eta\to C(\W,\s_n(\K)),\;K\mapsto M^+_K$,
are well defined and continuous. And the analogous statements hold for
$M^-_{0_{2n}}$.
\end{itemize}
\end{teor}
Now we introduce the concept of uniform weak disconjugacy.
\begin{defi}{\rm \label{2.defUWD}
Let $H$ take values in $\mathfrak{sp}(n,\R)$.
The family \eqref{2.hamil} of linear Hamiltonian systems is
{\em uniformly weakly disconjugate\/} (or {\em UWD\/} for short)
{\em on\/} $[0,\infty)$
(resp.~{\em on\/} $(-\infty,0]$)
if there exists $t_0\ge 0$ independent of $\w$ such that for every
nonzero solution $\bz(t,\w)=\lsm \bz_1(t,\w)\\\bz_2(t,\w)\rsm$ of the
systems corresponding to $\w$
with $\bz_1(0,\w)=\bcero$, there holds $\bz_1(t,\w)\ne\bcero$ for all
$t>t_0$ (resp.~$\bz_1(t,\w)\ne\bcero$ for all $t<-t_0$).
}\end{defi}
The results summarized in the next remarks can be found in Chapter 5
of \cite{jonnf}.
\begin{nota}\label{2.notasUWD}
Let us assume that $H_3\ge 0$.
\par
(a)~The uniform weak disconjugacy (also UWD for short) at $+\infty$ of the
family \eqref{2.hamil} is equivalent to the
UWD at $-\infty$: see Theorem 5.17 of \cite{jonnf}. We will
simply say that the family is UWD.
\par
(b)~If the family \eqref{2.hamil} is UWD, then there exist
{\em uniform principal solutions at $\pm\infty$},
$\lsm L_1^\pm(t,\w)\\L_2^\pm(t,\w)\rsm$.
They are real $2n\times n$ matrix-valued solutions of \eqref{2.hamil} satisfying the
following properties: for all $t\in\R$ and $\w\in\W$,
the matrices $L_1^\pm(t,\w)$ are nonsingular
and $\lsm L_1^\pm(t,\w)\\L_2^\pm(t,\w)\rsm$ represent Lagrange planes;
and for all $\w\in\W$,
\[
 \lim_{\pm t\to\infty}\left(\int_0^t
 (L_1^\pm)^{-1}(s,\w)\,H_3(\ws)\, ((L_1^\pm)^T)^{-1}(s,\w)\,ds\right)^{-1}\!=0_n\,.
\]
\par
(c)~If the matrix-valued functions
$\lsm L_1^\pm(t,\w)\\L_2^\pm(t,\w)\rsm$ are uniform principal solutions
at $\pm\infty$, then the real matrix-valued functions
$N^\pm\colon\W\to\s_n(\R)\,,\;\w\mapsto N^\pm(\w):=
L_2^\pm(0,\w)\,(L_1^\pm(0,\w))^{-1}$ are unique. They are called {\em principal
functions of\/} \eqref{2.hamil}, and they are globally defined
solutions along the flow
(see Remark~\ref{2.notasED}(d)) of the Riccati equation \eqref{2.riccati}.
\end{nota}
Many properties relating the ED of the family \eqref{2.hamil}
to its UWD will be used in the proofs of the results of Section~\ref{sec4}.
We will list them at the end of this section, pointing out where the reader
can find the corresponding proofs. Now
we formulate and prove a new lemma
concerning this relation for a particular type of family \eqref{2.hamil}.
Recall that a function $\wit M\colon\W\to\s_n(\R)$ is a
{\em globally defined solution along the flow of \eqref{2.riccati}}  if
the map $t\mapsto\wit M(\wt)$ is a globally defined solution of
the equation for all $\w\in\W$.
\begin{lema}\label{2.lemaRic}
Let $H$ take values in $\mathfrak{sp}(n,\R)$, and
let us suppose that $H_2>0$ and $H_3>0$. Then,
\begin{itemize}
\item[\rm(i)] the family of systems \eqref{2.hamil} is UWD and has ED,
the Weyl functions globally
exist and agree with the principal functions, and they satisfy
$\mp M^\pm>0$.
\item[\rm(ii)] If the function $\wit M\colon\W\to\s_n(\R)$ is continuous and
a globally defined
solution along the flow of \eqref{2.riccati} with $\wit M\ge 0$ (resp.~with $\wit M\le 0$), then
$\wit M>0$ (resp.~$\wit M<0$).
\end{itemize}
\end{lema}
\begin{proof}
(i) This assertion is proved by
Proposition 5.64(ii) of \cite{jonnf}, since the conditions
$H_2>0$ and $H_3>0$ guarantee conditions D2 and D2$^*$ required in that result:
see Remark 5.19 and the comments previous to Proposition 5.64 (also in \cite{jonnf}).
\smallskip\par
(ii) Let us denote $h(\w,M):=-MH_3(\w)M-MH_1(\w)-H_1^T(\w)M+H_2(\w)$, and represent
by $M(t,\w,M_0)$ the maximal solution of \eqref{2.riccati}
(i.e., of $M'=h(\wt,M)$) with $M(t,\w,M_0)=M_0$.
Since, by (i), $M^+(\w)<0<M^-(\w)$, then the monotonicity properties
of the Riccati equation (see Theorem 1.54 of \cite{jonnf})) ensure that
$M^+(\wt)\le M(t,\w,0_n)\le M^-(\wt)$ for $t$ in the
interval of definition of $M(t,\w,0_n)$, so that this interval is $\R$
(see e.g.~Remark 1.43 of \cite{jonnf}).
Since $H_2>0$, we can take $\ep>0$ such that
$h(\w,M(0,\w,0_n))=h(\w,0_n)=H_2(\w)>\ep I_n$. The compactness of $\W$
allows us to find $t_0>0$ such that $h(\wt,M(t,\w,0_n))\ge \ep I_n$ for
all $\w\in\W$ and $t\in[-t_0,t_0]$, which ensures that
\begin{equation}\label{2.t0}
 M(-t_0,\w,0_n)<-\ep\,t_0\,I_n \quad\text{and}\quad
 M(t_0,\w,0_n)>\ep\,t_0\,I_n \quad\text{for all $\w\in\W$}\,.
\end{equation}
Now assume that the function $\wit M$ of (ii) satisfies $\wit M\ge 0$.
The monotonicity yields
$\wit M(\w)=M(t_0,\w{\cdot}(-t_0),\wit M(\w{\cdot}(-t_0)))
\ge M(t_0,\w{\cdot}(-t_0),0_n)>0$, where
$t_0$ satisfies \eqref{2.t0}. The argument is analogous if $\wit M\le 0$.
\end{proof}
The last fundamental concept required for the proofs of the main results
is that of rotation number with respect to a
given $\sigma$-ergodic measure. Among the many equivalent definitions
for this quantity, we give one which extends that which is possibly the best known
in dimension 2. Recall that $U(t,\w)=\lsm
U_1(t,\w)&U_3(t,\w)\\ U_2(t,\w)&U_4(t,\w)\rsm$ is the matrix-valued
solution of \eqref{2.hamil} with $U(0,\w)=I_{2n}$. And $\arg\colon\C\to\R$
holds for the continuous branch of the argument of a complex number satisfying
$\arg 1=0$.
\begin{defi}\label{2.defrot}
Let $m_0$ be a $\sigma$-ergodic measure on $\W$.
The {\em rotation number of the family \eqref{2.hamil}
with respect to $m_0$} is the value of
\[
 \lim_{t\to\infty}\frac{1}{t}\:\arg\det (U_1(t,\w)-iU_2(t,\w))
\]
for $m_0$-a.a.~$\w\in\W$, which exists, is finite and common.
\end{defi}
The proof that this definition is correct can be found
in Chapter 2 of \cite{jonnf}, where the interested
reader will also find many other (equivalent) definitions
for the rotation number of different nature as well as an
exhaustive description of its properties.
\par
With the aim of simplifying the proofs of the main results of Section \ref{sec4},
we list now some properties relating exponential dichotomy, uniform weak
disconjugacy, and rotation number which we will use, and we indicate
where to find their proofs, all of them in \cite{jonnf}.
Note that the statements here lusted are not the optimal one,
but those which we will use; and that the list is far away from exhaustive.
More properties concerning the coincidence between Weyl functions and principal
functions can be found in \cite{hize}.
\par
In this list of properties we will use repeatedly the fact that $H_3>0$
is stronger than the conditions D1 and D2 of Chapter 5 of \cite{jonnf}
(see Remark 5.19 of \cite{jonnf}), which
are required in several of the results we make reference to. Also, the
matrix valued $F$ of \eqref{2.hamil} is supposed to take values in
$\mathfrak{sp}(n\R)$.
\begin{itemize}
\item[\hypertarget{p1}\text{\bf p1.}] Suppose that $H_3>0$ and $H_2\ge 0$. Then the
family \eqref{2.hamil} is uniformly weakly disconjugate. This assertion
is proved in Proposition 5.27 of \cite{jonnf}.
\item[\hypertarget{p2}\text{\bf p2.}] Suppose that $H_3>0$ and that the Weyl function
$M^+$ globally exists. Then the
family \eqref{2.hamil} is uniformly weakly disconjugate. This property
follows from Theorem 5.17 of \cite{jonnf}, since $H_3>0$,
and the global existence of $M^+$ is stronger
than the remaining required condition, D3.
\item[\hypertarget{p3}\text{\bf p3.}] Suppose that $H_3\ge 0$ and that the
family \eqref{2.hamil} is uniformly weakly disconjugate. Then
the principal functions satisfy $N^+\le N^-$: see Theorem 5.43
of \cite{jonnf}. In addition, $N^+<N^-$ if and only if the
family \eqref{2.hamil} has exponential dichotomy,
in which case the Weyl functions are globally defined and satisfy
$M^\pm=N^\pm$: see Theorem 5.58 of~\cite{jonnf}.
\item[\hypertarget{p4}\text{\bf p4.}] Suppose that
$H_3\ge 0$ and the family \eqref{2.hamil} is uniformly weakly disconjugate.
Let $\wit M\colon\W\to\s_n(\R)$ be a globally defined solution along the
flow of the Riccati equation \eqref{2.riccati}. Then
$N^+\le \wit M\le N^-$. This fact is proved in
Theorem 5.48 of \cite{jonnf}.
\item[\hypertarget{p5}\text{\bf p5.}] Suppose that
the families $\bz'=H^1(\wt)\,\bz$ and $\bz'=H^2(\wt)\,\bz$
satisfy $H^1_3>0$ and $H^2_3>0$. Suppose also that
the family $\bz'=H^2(\wt)\,\bz$ is uniformly weakly disconjugate,
and that $JH^1\le JH^2$.
Then the family $\bz'=H^1(\wt)\,\bz$ is uniformly weakly disconjugate, and
the corresponding principal functions $N_1^\pm$ and $N_2^\pm$
satisfy $N_1^+\le N_2^+\le N_2^-\le N_1^-$. This assertion is
an immediate consequence of Proposition 5.51
of \cite{jonnf}.
\item[\hypertarget{p6}\text{\bf p6.}]
Let $\G\colon\W\to\s_{2n}(\R)$ be continuous and satisfy $\G\ge 0$,
and let us consider the families
$\bz'=(H(\wt)+\alpha J^{-1}\G(\wt))\,\bz$ for $\alpha\in\R$.
Then the rotation
number increases as $\alpha$ increases. The proof of this well-known assertion
can be found in Proposition 2.33 of \cite{jonnf}.
\item[\hypertarget{p7}\text{\bf p7.}]
Suppose that $\W=\Supp m_0$ for a $\sigma$-ergodic measure $m_0$, and let $\G\colon\W\to\s_{2n}(\R)$ be continuous and satisfy $\G>0$.
Let $\mI\subseteq\R$ be an open
interval, and let us consider the families
$\bz'=(H(\wt)+\alpha J^{-1}\G(\wt))\,\bz$ for $\alpha\in\mI$.
Then these families have exponential dichotomy over $\W$
for all $\alpha\in \mI$ if and only if the rotation number
with respect to $m_0$ is constant on $\mI$. This assertion
(as a matter of fact, a more general one) is one of the main
results of \cite{jone}, and a very detailed proof is given in Theorem 3.50 of \cite{jonnf}.
\item[\hypertarget{p8}\text{\bf p8.}]
The \lq\lq only if" part of the previous property can be extended
to more general situations. Let $m_0$ be a $\sigma$-ergodic measure on $\W$.
Suppose that the family \eqref{2.hamil} has exponential dichotomy, so that
Theorem \ref{2.teorSSPT}  provides a neighborhood of $H$
in $C(\W,\mathfrak{sp}(n,\R))$ such that the corresponding families
of linear Hamiltonian
systems have exponential dichotomy. Then the rotation number with respect to
$m_0$ is common
for all these families. This assertion follows, for instance, from
Theorems 2.28 and 2.25 of \cite{jonnf}.
\item[\hypertarget{p9}\text{\bf p9.}]
Suppose that $\W=\Supp m_0$ for a $\sigma$-ergodic measure $m_0$, and that
$H_3>0$. Then the family \eqref{2.hamil} is uniformly weakly disconjugate
if and only if its rotation number with respect to $m_0$ is 0. This is proved
in Theorem 5.67 of \cite{jonnf}.
\item[\hypertarget{p10}\text{\bf p10.}]
Suppose that the family \eqref{2.hamil} satisfies the nonoscillation condition, and
that $H_3\ge 0$. Then its rotation number with respect to any ergodic measure is 0.
This assertion can be proved using Proposition 5.8 of \cite{jonnf} to check that
all the systems of the family are nonoscillatory at $\infty$, and then applying
Proposition 5.65 of \cite{jonnf}.
\end{itemize}
\section{Global existence of Weyl functions}\label{sec3}
Let $(\W,\sigma)$ be a real continuous global flow on a
compact metric space, and let us denote $\wt=\sigma(t,\w)$.
Let us consider a continuous matrix-valued function
$H\colon\W\to\mathfrak{sp}(n,\R)$, with $H=\lsm H_1&H_3\\H_2&-H_1^T\rsm$,
which provides the family of linear Hamiltonian systems
systems
\begin{equation}\label{3.hamil}
\bz'=H(\wt)\,\bz
\end{equation}
for $\w\in\W$.
Given a continuous matrix-valued function $\Delta\colon\W\to\s_n(\R)$,
we consider the perturbed families of Hamiltonian systems
\begin{equation}\label{3.hamilper}
\bz'=H^\lambda(\wt)\,\bz\,,\quad \text{where }\;
 H^\lambda(\w):=\left[\begin{array}{cc} H_1(\w)&
 H_3(\w)+\lambda\Delta(\w)\\
 H_2(\w)&-H_1^T(\w)\end{array}\right]
\end{equation}
for $\w\in\W$. The parameter $\lambda$ varies in $\C$, and
we will use the notation \eqref{3.hamilper}$^\lambda$
to make reference to a particular value of $\lambda$. Obviously,
\eqref{3.hamilper}$^0$ agrees with \eqref{3.hamil}.
\par
We will analyze in this section two different scenarios with $\Delta>0$
under which, if $\lambda\in\C-\R$,
the families \eqref{3.hamilper}$^\lambda$ have ED and there exist
both Weyl functions, which we will denote by $M^\pm(\w,\lambda)$.
These results will be used in the proof of the main results in
Section~\ref{sec4}, but have independent
interest. In particular, Theorem~\ref{3.teornoatki} analyzes
this question in the absence of the so-called {\em Atkinson condition}
(see \eqref{3.atkgen}), which is usually required to guarantee the
mentioned properties. It is also important to
emphasize that, in the two cases, the Weyl functions will be
Herglotz functions on the complex upper and lower half-planes
for each fixed $\w\in\W$. As usual, we represent
$\C^\pm:=\{\lambda\in\C\,|\;\pm\Ima\lambda>0\}$.
\begin{defi}{\rm \label{3.defherg}
A symmetric matrix-valued function $M$ defined on $\C^+$ or
$\C^-$ is {\em Herglotz\/}
if it is holomorphic
and $\Ima M(\lambda)$ is either positive semidefinite
or negative semidefinite on the whole half-plane.
}\end{defi}
Let us represent by $\bz(t,\w,\bz_0)=\lsm\bz_1(t,\w,\bz_0)\\
\bz_2(t,\w,\bz_0)\rsm$ the solution of the system~\eqref{3.hamil}
corresponding to $\w$ which satisfies $\bz(0,\w,\bz_0)=\bz_0$.
The first result (which as a matter of fact is not new: see its proof)
is formulated under the next Atkinson-type condition on $\Delta$.
\begin{hipo}\label{3.hipoatki}
$\Delta\ge 0$, and each minimal
subset of $\W$ contains at least one point $\w_0$ such that
\begin{equation}\label{3.atk}
 \int_{-\infty}^{\infty} \n{\Delta(\w_0{\cdot}t)\,\bz_2(t,\w_0,\bz_0)}^2\, dt > 0
 \quad \mbox{whenever $\bz_0\in\C^{\,2n}\!-\!\{\bcero\}$}\,.
\end{equation}
\end{hipo}
\begin{teor}\label{3.teoratki}
Suppose that Hypothesis~\ref{3.hipoatki} holds.
\begin{itemize}
\item[\rm(i)] If $\Ima\lambda\ne 0$, then
the family~\eqref{3.hamilper}$^\lambda$ has exponential dichotomy.
\item[\rm(ii)] If $\Ima\lambda\ne 0$, then there globally exist the Weyl functions
$M^\pm(\w,\lambda)$. In addition, the maps
$M^\pm\colon\W\times(\C-\R)\to\s_n(\C),\;(\w,\lambda)\mapsto M^\pm(\w,\lambda)$
are jointly continuous, satisfy
$\pm\Ima\lambda\,\Ima M^{\pm}(\w,\lambda)>0$, and
are holomorphic on $\C-\R$ for each $\w\in\W$ fixed. In particular,
they are Herglotz functions on $\C^+$ and $\C^-$.
\end{itemize}
\end{teor}
\begin{proof}
In the general case of a perturbed Hamiltonian system of the form
\begin{equation}\label{3.hamilpergamma}
 \bz'
 =(H(\wt)+\lambda\J\,\G(\wt))\,\bz
\end{equation}
for a continuous
perturbation matrix-valued function $\Gamma\colon\W\to\s_{2n}(\R)$,
all the conclusions of Theorem~\ref{3.teoratki} hold under the
following general Atkinson condition: $\Gamma\ge 0$, and each
minimal subset of $\W$ contains at least one point $\w_0$ such that
\begin{equation}\label{3.atkgen}
 \int_{-\infty}^{\infty} \n{\G(\w_0{\cdot}t)\,\bz(t,\w_0,\bz_0)}^2\, dt > 0
 \quad \mbox{whenever $\bz_0\in\C^{\,2n}\!-\!\{\bcero\}$}\,.
\end{equation}
This assertion is originally proved in \cite{john1},
and a very detailed proof can be found in Theorems 3.8 and 3.9 of \cite{jonnf}.
It is also clear that in the case of~\eqref{3.hamilper},
$H^\lambda=H+\lambda\J\,\G$ for $\G:=\lsm 0_n&0_n\\0_n&\Delta\rsm$, and hence
that Hypothesis~\ref{3.hipoatki} is the general one applied to the particular case.
\end{proof}
\begin{nota}\label{3.notaatki}
(a)~It is very easy to check that a function
$\bz(t,\w)=\lsm\bz_1(t,\w)\\\bcero\rsm$ solves the system
\eqref{3.hamilper}$^{\lambda_0}$ corresponding to $\w$
for a $\lambda_0\in\C$  if and only if it solves
the system \eqref{3.hamilper}$^\lambda$ corresponding to the same
$\w$ for all $\lambda\in\C$: both conditions are equivalent
to saying that $\bz_1'(t,\w)=H_1(\wt)\,\bz_1(t,\w)$
and $\bcero=H_2(\wt)\,\bz_1(t,\w)$, so that $\lambda$ plays no role.
\par
(b)~Let us assume that $\Delta>0$. Lemma 3.6(iv) of \cite{jonnf}
ensures that $\Delta$ satisfies
Hypothesis~\ref{3.hipoatki} (or, equivalently,
$\G:=\lsm 0_n&0_n\\0_n&\Delta\rsm$ satisfies \eqref{3.atkgen}) if
and only if \eqref{3.atk} holds {\em for all} $\w\in\W$. This ensures
that $\Delta>0$ does not
satisfy Hypothesis~\ref{3.hipoatki} if and only if there exist
$\w\in\W$ and $\bz_0\in\C^{2n}-\{\bcero\}$ such that
$\bz(t,\w,\bz_0)=\lsm\bz_1(t,\w,\bz_0)\\\bcero\rsm$ for all $t\in\R$.
(This is for instance the case when
$\Delta=I_n$ and $H=\lsm I_n&0_n\\0_n&I_n\rsm$.)
According to the previous remark, $\Delta>0$
does not satisfy Hypothesis~\ref{3.hipoatki}
if and only if there exists a point $\w\in\W$ and a $\lambda_0\in\C$
such that the system~\eqref{3.hamilper}$^{\lambda_0}$
admits a nontrivial solution $\bz(t,\w)=\lsm\bz_1(t,\w)\\\bcero\rsm$, in
which case this function solves the system~\eqref{3.hamilper}$^\lambda$
for the same $\w$ and all $\lambda\in\C$.
\par
(c)~As a matter of fact, $\Delta>0$ does not
satisfy Hypothesis~\ref{3.hipoatki} if and only if there exist
a minimal subset $\mM\subseteq\W$ such that all the systems
\eqref{3.hamilper}$^{\lambda}$ corresponding to $\w\in\mM$
admit a nontrivial solution
$\bz(t,\w)=\lsm\bz_1(t,\w)\\\bcero\rsm$ (common for all $\lambda\in\C$).
\par
(d)~Let $U_{H_1}(t,\w)$ represent the matrix-valued solution of
$\bz_1'=H_1(\wt)\,\bz_1$ with $U_{H_1}(0,\w)=I_n$.
Note that $\bz(t)=\lsm\bz_1(t)\\\bcero\rsm\not\equiv\bcero$
is a solution of the system~\eqref{3.hamil} corresponding to a point $\w\in\W$
if and only $\bz_1(0)=\bz_1^0\ne\bcero$ with
$H_2(\wt)\,U_{H_1}(t,\w)\,\bz_1(0)=\bcero$
for any $t\in\R$, in which case $\bz_1(t)=U_{H_1}(t,\w)\,\bz_0^1$.
Since $\bz_1(t)\ne\bcero$ for all $t\in\R$, the
existence of such a solution ensures that
$\det H_2(\wt)=0$ for all $t\in\R$. By continuity, there must exist a minimal set
(contained in the omega limit of $\w$ for the base flow)
on which $\det H_2$ vanishes identically.
\par
(e)~Note finally that a continuous map
$\G\colon\W\to\s_{2n}(\R)$ with $\G>0$ satisfies \eqref{3.atkgen}
for all $\w_0\in\W$,
and hence the results of \cite{john1} (see Theorems 3.8 and 3.9 of \cite{jonnf})
ensure that all the conclusions of Theorem~\ref{3.teoratki}
apply to the family \eqref{3.hamilpergamma}.
We will use this property later.
(As a matter of fact, it is enough that each minimal
subset of $\W$ contains a point $\w$ with
$\G(\w)>0$.) This is an important difference with the
case of $\Delta$: the second condition of Hypothesis \ref{3.hipoatki}
is not guaranteed by $\Delta>0$.
\end{nota}
The previous Remarks \ref{3.notaatki}(b),(c)\&(d) describe possible situations
in which a continuous matrix-valued function $\Delta\colon\W\to\s_n(\R)$
with $\Delta>0$ may not satisfy Hypothesis \ref{3.hipoatki}. The next result
will also prove the occurrence of ED and the global existence of Weyl functions for
$\lambda$ outside the real line, under a different condition.
The point $\lambda_0\in\C$ appearing in its hypothesis can of course be real.
\begin{hipo}\label{3.hiposED}
$\Delta>0$, and there exists $\lambda_0\in\C$ such that the
family~\eqref{3.hamilper}$^{\lambda_0}$
has exponential dichotomy.
\end{hipo}
\begin{teor}\label{3.teornoatki}
Suppose that Hypothesis~\ref{3.hiposED} holds.
\begin{itemize}
\item[\rm(i)] If $\Ima\lambda\ne 0$, then
the family~\eqref{3.hamilper}$^\lambda$ has exponential dichotomy.
\item[\rm(ii)] If $\Ima\lambda\ne 0$, there globally exist the Weyl functions
$M^\pm(\w,\lambda)$. In addition, the maps
$M^\pm\colon\W\times(\C-\R)\to\s_n(\C),\;(\w,\lambda)\mapsto M^\pm(\w,\lambda)$
are jointly continuous, satisfy
$\pm\Ima\lambda\,\Ima M^{\pm}(\w,\lambda)\ge 0$, and
are holomorphic on $\C-\R$ for each $\w\in\W$ fixed. In particular,
they are Herglotz functions on $\C^+$ and $\C^-$.
\item[\rm(iii)] If Hypothesis~\ref{3.hipoatki} does not hold,
there exists a minimal subset $\mM\subseteq\W$ such that
either $\det M^+(\w,\lambda)=0$ for all $\w\in\mM$ and all $\lambda\in\C-\R$
or $\det M^-(\w,\lambda)=0$ for all $\w\in\mM$
and all $\lambda\in\C-\R$.
\end{itemize}
\end{teor}
\begin{proof}
The arguments that we will use adapt those of the proof
of the result corresponding to the
Atkinson condition \eqref{3.atkgen}
(see again Theorem 3.8 of \cite{jonnf}).
\par
(i) We fix $\w\in\W$ and $\lambda_0\in\C-\R$, and represent
$\n{\bz}_{\Delta_t}=(\bz^*\Delta(\wt)\,\bz)^{1/2}$.
The main step of the proof shows that the
system of the family~\eqref{3.hamilper}$^{\lambda_0}$ corresponding
to our choice of $\w$ does not admit a
nonzero bounded solution. We define the functional $\mL_\w^{\lambda_0}$ as
\[
 (\mL_\w^{\lambda_0}\bz)(t)=J\bz'(t)-JH^{\lambda_0}(\wt)\,\bz(t)\,,
\]
and observe that, for any solution $\bz=\lsm\bz_1\\\bz_2\rsm$ of the
system~\eqref{3.hamilper}$^{\lambda_0}$ corresponding to $\w$, we have
$\mL_\w^{\lambda_0}\bz\equiv\bcero$ and hence
\begin{equation}\label{3.ED2}
\begin{split}
 0&=\int_a^b \left(\bz^*(t)\,(\mL_\w^{\lambda_0}\bz)(t)-
 (\mL_\w^{\lambda_0}\bz)^*(t)\,\bz(t)\right)\,dt\\
 &=\left.\bz^*(t)J\bz(t)\right|_{t=a}^{t=b}-
 2\,i\Ima\lambda_0\int_a^b\n{\bz_2(t)}_{\Delta_t}^2\,dt
\end{split}
\end{equation}
whenever $a<b$. Let us assume for contradiction that there exists
a bounded solution $\bz(t,\w,\bz_0)=
\lsm\bz_1(t,\w,\bz_0)\\\bz_2(t,\w,\bz_0)\rsm$ of~\eqref{3.hamilper}$^{\lambda_0}$.
Then~\eqref{3.ED2} ensures that
\[
 \int_\R \n{\bz_2(t,\w,\bz_0)}_{\Delta_t}^2dt<\infty\,,
\]
which provides an increasing sequence $(t_m)\uparrow\infty$ such that
\[
 \int_{t_m}^{t_{m}+1}\n{\bz_2(t,\w,\bz_0)}_{\Delta_t}^2 dt<\frac{1}{m}
\]
for every $m\in\N$. The compactness of $\W$ and the
boundedness of $(\wit\bz(t_m))$ provide
a subsequence $(t_j)$ and points $\wit\w\in\W$ and $\wit\bz_0\in\C^{\,2n}$
such that $\wit\w=\lim_{j\to\infty}\wt_j$ and
$\wit\bz_0=\lim_{j\to\infty}\bz(t_j,\w,\bz_0)$.
Consequently,
\[
 \bz(t,\wit\w,\wit\bz_0)=\lim_{j\to\infty}\bz(t,\wt_j,
 \bz(t_j,\w,\bz_0))\,.
\]
Hence, since
\[
 \frac{1}{j}>\int_{t_j}^{{t_j}+1}\n{\bz_2(t,\w,\bz_0)}_{\Delta_t}^2dt
 =\int_0^1\n{\bz_2(t,\wt_j,\bz(t_j,\w,\bz_0))}_{\Delta_t}^2dt\,,
\]
we find that
\[
 \int_0^1\n{\bz_2(t,\wit\w,\wit\bz_0)}_{\Delta_t}^2dt=0\,,
\]
which, since $\Delta>0$, ensures that $\bz_2(0,\wit\w,\wit\bz_0)=\bcero$.
In other words,
\[
 \lim_{j\to\infty}\bz_2(t_j,\w,\bz_0)=\bcero\,.
\]
A symmetric argument provides a sequence $(s_j)\downarrow-\infty$ such that
\[
 \lim_{j\to\infty}\bz_2(s_j,\w,\bz_0)=\bcero\,.
\]
Therefore, applying~\eqref{3.ED2} to each interval $[s_j,t_j]$
and taking limits as $j\to\infty$
yields
\[
 \int_{-\infty}^\infty\n{\bz_2(t,\w,\bz_0)}_{\Delta_t}^2dt=0\,,
\]
and since $\Delta>0$ it follows that $\bz_2(t,\w,\bz_0)\equiv\bcero$. This means that
$\bz(t,\w,\bz_0)=\lsm\bz_1(t,\w,\bz_0)\\\bcero\rsm$
is a bounded solution of the system~\eqref{3.hamilper}$^{\lambda_0}$
corresponding to $\w$. But it is immediate to check that it also solves
the system the system~\eqref{3.hamilper}$^\lambda$
corresponding to this $\w$ for all $\lambda\in\C$, including $\lambda=\lambda_0$.
The contradiction has been reached: according to Remark~\ref{2.notasED}(a),
the existence of this nontrivial bounded solution precludes the exponential
dichotomy of \eqref{3.hamilper}$^{\lambda_0}$, assumed from the beginning.
\par
(ii) We take $\lambda\in\C-\R$, so that \eqref{3.hamilper}$^\lambda$ has ED.
Let $L^\pm_\lambda$ be the invariant subbundles appearing in Definition
\ref{2.defED}, and let $l^{\pm}(\w,\lambda)$ be the corresponding sections,
given by \eqref{2.defl}. We also take
$\w\in\W$, and assume for contradiction
that there exists $\bz^0=\lsm\bcero\\\bz_2^0\rsm\in l^+(\w,\lambda)$.
Applying~\eqref{3.ED2} to the solution $\bz(t,\w,\bz^0)$ on intervals $[0,t]$ for $t>0$,
and having in mind that $\lim_{t\to\infty}\bz(t,\w,\bz^0)=\bcero$
(see \eqref{2.deflpm}), we obtain
$\int_0^\infty\n{\bz_2(t,\w,\bz_0)}_{\Delta_t}^2dt=0$. This ensures that
$\bz_2(t,\w,\bz^0)=\bcero$ for any $t\ge 0$. In particular, $\bz_2^0=\bcero$,
so that $\bz^0=\bcero$. Hence $l^+(\w,\lambda)$ contains no nontrivial vectors of the
form $\lsm\bcero\\\bz_2\rsm$, and thus it can be represented
by $\lsm I_n\\M^+(\w,\lambda)\rsm$, where $M^\pm(\w,\lambda)$ is symmetric.
An analogous argument shows the global existence of
$M^-(\w,\lambda)$. The continuity of the
map $M^+(\w,\lambda)$ on $\W\times(\C-\R)$ is guaranteed by
Theorem~\ref{2.teorSSPT}. The holomorphic character
of $\lambda\mapsto M^\pm(\w,\lambda)$ outside the real axis
can be proved repeating the argument of the proof of Theorem 3.9 of \cite{jonnf}.
\par
It remains to prove that $\pm\Ima\lambda\Ima M^\pm(\w,\lambda)\ge 0$.
To this end, we consider
the new auxiliary perturbed systems $\bz'=H^\lambda_k(\wt)\,\bz$ with
$H^\lambda_k=H+\lambda\J\,\G_k$ for $\G_k=\lsm (1/k)\,I_n&0_n\\0_n&\Delta\rsm$.
Since $\G_k>0$ for $k\ge 1$, it satisfies the general Atkinson condition and
the conclusions of Theorem~\ref{3.teoratki}
hold (see Remark~\ref{3.notaatki}(e)); thus,
if $\lambda\in\C-\R$, then there exist the corresponding Weyl functions
$M^\pm_k(\w,\lambda)$ and they satisfy $\pm\Ima\lambda\,M^\pm_k(\w,\lambda)>0$.
Fix $\lambda\notin\R$, and note that $\lim_{k\to\infty} H^\lambda_k=H^\lambda$
uniformly on $\W$. Therefore, Theorem \ref{2.teorSSPT}
ensures that $\lim_{k\to\infty}M^\pm_k(\w,\lambda)=
M^\pm(\w,\lambda)$. Consequently, $\pm\Ima\lambda\Ima M^\pm(\w,\lambda)\ge 0$,
which completes the proof of (ii).
\par
(iii) Since condition \eqref{3.atk} does not hold, there exists a
minimal subset $\mM\subseteq\W$ and a nontrivial solution
of the form $\bz(t,\w)=\lsm\bz_1(t,\w)\\\bcero\rsm$
of the systems corresponding to $\w\in\mM$
of the families \eqref{3.hamilper}$^\lambda$
for all $\lambda\in\C$: see Remark~\ref{3.notaatki}(c).
Let us fix $\wit\lambda\in\C-\R$, so that the functions
$M^\pm(\w,\lambda)$ globally exist. Let us also fix $\wit\w\in\mM$.
\par
If $\bz(0,\wit\w)=\lsm\bz_1(0,\wit\w)\\\bcero\rsm$ belongs
to $l^+(\wit\w,\wit\lambda)$, then $\bz(t,\wit\w)=\lsm\bz_1(t,\wit\w)\\\bcero\rsm$
belongs to $l^+(\wit\w{\cdot}t,\wit\lambda)$ for all $t\in\R$. Since
$l^+(\wit\w{\cdot}t,\wit\lambda)$ can be represented by
$\lsm I_n\\M^+(\wit\w{\cdot}t,\wit\lambda)\rsm$, we have
$\det M^+(\wit\w{\cdot}t,\wit\lambda)=0$. The continuity of $M^+$ and the
minimality of $\mM$ ensure that $\det M^+(\w,\wit\lambda)=0$ for
all $\w\in\mM$.
\par
Note now that $\bz(0,\wit\w)=\lsm\bz_1(0,\wit\w)\\\bcero\rsm$ belongs
to $l^+(\wit\w,\lambda)$ for all $\lambda\in\C-\R$, as we deduce from
Remark~\ref{3.notaatki}(a) and from the characterization \eqref{2.deflpm}
of the Lagrange plane. Therefore the previous argument can be repeated
in order to show that $\det M^+(\w,\lambda)=0$ for
all $\w\in\mM$ and all $\lambda\in\C-\R$.
\par
In the remaining cases, $\bz(0,\wit\w)=\bz_0^++\bz_0^-$ with
$\bz^\pm_0\in l^\pm(\wit\w,\wit\lambda)$ and $\bz_0^-\ne\bcero$, and it follows from
\eqref{2.noacotada} that
\begin{equation}\label{3.liminf}
 \lim_{t\to\infty}\n{\bz(t,\wit\w)}=\infty\,.
\end{equation}
We denote $\bz^\pm(t,\wit\w)=\bz(t,\wit\w,\bz_0^\pm)$ and observe that
$\bz(t,\wit\w)=\bz^+(t,\wit\w)+\bz^-(t,\wit\w)$ and
$(\wit\w{\cdot}t,\bz^\pm(t,\wit\w))\in L^\pm_{\wit\lambda}$
for all $t\in\R$.
Now we take $\w\in\mM$ and
choose $(t_m)\uparrow\infty$ with $\lim_{m\to\infty}\wit\w{\cdot}t_m=\w$
and such that there exists $\bz^*:=
\lim_{m\to\infty}\bz(t_m,\wit\w)/\n{\bz(t_m,\wit\w)}$.
It follows from~\eqref{2.deflpm} and~\eqref{3.liminf}
that $\lim_{m\to\infty}\bz^+(t_m,\wit\w)/\n{\bz(t_m,\wit\w)}=\bcero$, so that
$\lim_{m\to\infty}\bz^-(t_m,\wit\w)/\n{\bz(t_m,\wit\w)}=
\lim_{m\to\infty}\bz(t_m,\wit\w)/\n{\bz(t_m,\wit\w)}=\bz^*$.
The closed character of $L^-_{\wit\lambda}$ and the fact that
$\bz^-(t_m,\wit\w)/\n{\bz(t_m,\wit\w)}\in l^-(\wit\w{\cdot}t_m,\wit\lambda)$
ensure that $\bz^*\in l^-(\w,\wit\lambda)$.
In addition, $\bz^*$ is the initial data of a solution of \eqref{3.hamilper}$^\lambda$
of the form $\bz^*(t,\w)=\lsm\bz^*_1(t,\w)\\\bcero\rsm$ for all
$\lambda\in\C-\R$ (see again Remark~\ref{3.notaatki}(a)),
and hence $\bz^*=\lsm\bz_1^*\\\bcero\rsm\in l^-(\w,\lambda)$ for all
$\lambda\in\C-\R$. Therefore, $\det M^-(\w,\lambda)=0$ for all $\lambda\in\C-\R$.
And also for all $\w\in\mM$, since $\w$ has been arbitrarily chosen.
This completes the proof.
\end{proof}
The statement of the previous theorem and the proof of its point (i)
prove the next result.
\begin{coro}\label{3.coronoatki}
Suppose that the continuous matrix-valued function $\Delta\colon\W\to\s_n(\R)$
takes positive definite values.
Then, there are two dynamical possibilities for
the families~\eqref{3.hamilper}$^\lambda$:
\begin{itemize}
\item[O1.] There exist $\lambda_0\in\C$ such that the
family~\eqref{3.hamilper}$^{\lambda_0}$ has exponential
dichotomy. In this case the families
\eqref{3.hamilper}$^{\lambda}$ have exponential dichotomy
for (at least) all $\lambda\in\C-\R$,
and the Weyl functions $M^\pm(\w,\lambda)$ globally exist for all $\lambda
\in\C-\R$ and are Herglotz functions.
\item[O2.] The family~\eqref{3.hamilper}$^\lambda$
does not have exponential dichotomy
for any $\lambda\in\C$. Equivalently, there exists a point $\w\in\W$ and
a $\lambda_0\in\C$ such that the system~\eqref{3.hamilper}$^{\lambda_0}$
corresponding to $\w$ admits a nonzero bounded solution of the form
$\bz(t,\w)=\lsm\bz_1(t,\w)\\\bcero\rsm$, in which case this function
solves the system~\eqref{3.hamilper}$^\lambda$
corresponding to $\w$ for all $\lambda\in\C$.
\end{itemize}
\end{coro}
Note that situation O2 is extremely non-persistent.
For instance, O1 holds in the following cases:
\begin{itemize}
\item[-] When $\Delta$ satisfies the Atkinson Hypothesis~\ref{3.hipoatki},
as Theorem~\ref{3.teoratki} ensures.
\item[-] When $\det H_2$ does not vanish identically on any minimal subset
$\mM\subset\W$: Remark~\ref{3.notaatki}(d) ensures that in this case
$\Delta$ satisfies the Atkinson Hypothesis~\ref{3.hipoatki}.
\item[-] If the $n$-dimensional family of systems
$\bz_1'=H_1(\wt)\,\bz_1$ has exponential
dichotomy, since any nonzero solution
$\lsm\bz_1(t,\w)\\\bcero\rsm$ of \eqref{3.hamilper} provides a nonzero
solution of $\bz_1(t,\w)$ of $\bz_1'=H_1(\wt)\,\bz_1$ which cannot be bounded
(see Remark~\ref{2.notasED}(a)).
\end{itemize}
\par
We conclude this section with another consequence of Theorem~\ref{3.teornoatki}
which concerns other type of perturbed systems, namely
\begin{equation}\label{3.hamilper2}
\bz'=\wit H^\lambda(\wt)\,\bz\,,\quad \text{where }\;
 \wit H^\lambda(\w):=\left[\begin{array}{cc} H_1(\wt)&
 H_3(\w)\\
 H_2(\w)+\lambda\Delta(\w)&-H_1^T(\w)\end{array}\right].
\end{equation}
\begin{coro}\label{3.coro2}
Suppose that the continuous matrix-valued function $\Delta\colon\W\to\s_n(\R)$
takes positive definite values.
Then, there are two dynamical possibilities for
the families~\eqref{3.hamilper2}$^\lambda$:
\begin{itemize}
\item[O1$^*$.] There exist $\lambda_0\in\C$ such that the
family~\eqref{3.hamilper2}$^{\lambda_0}$ has exponential dichotomy.
In this case the families
\eqref{3.hamilper}$^{\lambda}$ have exponential dichotomy
for (at least) all $\lambda\in\C-\R$.
\item[O2$^*$.] The family~\eqref{3.hamilper2}$^\lambda$
does not have exponential dichotomy
for any $\lambda\in\C$. Equivalently, there exists a point $\w\in\W$ and
a $\lambda_0\in\C$ such that the system~\eqref{3.hamilper2}$^{\lambda_0}$
corresponding to $\w$ admits a nonzero bounded solution of the form
$\bz(t,\w)=\lsm\bcero\\\bz_2(t,\w)\\\rsm$, in which case this function
solves the system~\eqref{3.hamilper2}$^\lambda$
corresponding to $\w$ for all $\lambda\in\C$.
\end{itemize}
\end{coro}
\begin{proof}
It is easy to check that
the change of variables $\bw=\lsm 0_n&I_n\\I_n&0_n\rsm\bz$ takes
\eqref{3.hamilper2} to
\begin{equation}\label{3.hamilper3}
 \bw'=\left[\begin{array}{cc} -H_1^T(\wt)&
  H_2(\wt)+\lambda\Delta(\wt)\\H_3(\wt)&H_1(\wt)\end{array}\right]\bw\,,
\end{equation}
which is in one of the situations described in Corollary~\ref{3.coronoatki}.
Obviously a nonzero bounded solution exists for one of the systems
of \eqref{3.hamilper3}$^\lambda$ if and only a nonzero bounded solution
exists for one of the systems of \eqref{3.hamilper2}$^\lambda$.
This fact allows us to deduce from Remark~\ref{2.notasED}(b) that the ED
holds or not simultaneously for \eqref{3.hamilper3}$^\lambda$ and
\eqref{3.hamilper2}$^\lambda$. Therefore, the assertions follow
from Corollary~\ref{3.coronoatki}.
\end{proof}
\begin{nota}
(a)~Note that a family of the type \eqref{3.hamilper2} arises when dealing with
the $n$-dimensional Schr\"{o}dinger family $\bx'+G(\wt)\,\bx=\lambda\Delta(\wt)$,
by taking $\bz=\lsm\bx\\\bx'\rsm$. It is known (and very easy to check) that,
in this case, the perturbation matrix $\Gamma=\lsm\Delta&0_n\\0_n&0_n\rsm$
satisfies the general Atkinson condition \eqref{3.atkgen}, so that
the statements of Theorem~\ref{3.teoratki} hold in this case. In particular,
the Schr\"{o}dinger case is in situation O1$^*$. But clearly the situation
that we consider in Corollary~\ref{3.coro2} is much more general.
\par
(b)~A linear Hamiltonian system admitting a nontrivial bounded solution
of the form $\bz(t)=\lsm\bcero\\\bz_2(t)\\\rsm$ on a
positive of negative half-line is called {\em abnormal at $+\infty$}
or {\em at $-\infty$}.
Note that, in situation O2$^*$, each one of the families
\eqref{3.hamilper2}$^\lambda$ has an abnormal system both
at $+\infty$ and at $-\infty$. The reader is referred to
\cite{reid3,reid8,sesi1,sesi2,sesi3,fjnn2,jnno2} and references
therein for an analysis of abnormal linear Hamiltonian systems.
\end{nota}
\section{Exponential dichotomy and nonoscillation condition
for parametric families}\label{sec4}
As in the previous section, $(\W,\sigma)$ is a real continuous global flow on a
compact metric space, and $\wt=\sigma(t,\w)$.
Given continuous functions
$H\colon\W\to\mathfrak{sp}(n,\R)$ with $H=\lsm H_1&H_3\\H_2&-H_1^T\rsm$
and $\Delta\colon\W\to\s_n(\R)$,
we consider the families of linear Hamiltonian systems
\begin{equation}\label{4.hamil}
 \bz'=H(\wt)\,\bz
\end{equation}
and
\begin{equation}\label{4.hamilperzl}
 \bz'=H^\lambda(\wt)\,\bz\,,
 \quad \text{where }\;H^\lambda(\w):=\left[\begin{array}{cc} H_1(\w)&
 H_3(\w)\\
 H_2(\w)-\lambda\Delta(\w)&-H_1^T(\w)\end{array}\right]
\end{equation}
for $\w\in\W$. The parameter $\lambda$ may vary in $\C$, although
our results will refer to real values of $\lambda$.
We will use the notation \eqref{4.hamilperzl}$^\lambda$ to refer to the
family corresponding to a particular value of $\lambda$.
Note that \eqref{4.hamilperzl}$^0$ and \eqref{4.hamil} coincide.
\par
The concepts of ED and NC appearing in the next set
of conditions, under which the results of this section will be obtained,
are given in Definitions \ref{2.defED} and \ref{2.defNC}.
\begin{hipo}\label{4.hipos}
$H_3\ge 0$, $\Delta>0$, and
the family \eqref{4.hamil} has exponential dichotomy and it satisfies
the nonoscillation condition.
\end{hipo}
The main goal of this section is to prove Theorem~\ref{4.teormejor},
in which statement a new set of
families of linear Hamiltonian system plays a role:
\begin{equation}\label{4.hamilperzle}
 \bz'=H_\ep^\lambda(\wt)\,\bz\,,\quad
 \text{where }\; H_\ep^\lambda(\w):=\left[\begin{array}{cc} H_1(\w)&
 H_3(\w)+\ep I_n\\
 H_2(\w)-\lambda\Delta(\w)&-H_1^T(\w)\end{array}\right]\bz
\end{equation}
for $\ep\in\R$. We will use the notation \eqref{4.hamilperzle}$^\lambda_\ep$
to refer to the family corresponding to particular values of $\lambda$
and $\ep$. Note that \eqref{4.hamilperzle}$^\lambda_0$ and
\eqref{4.hamilperzl}$^\lambda$ agree.
The systems of the family \eqref{4.hamilperzle} will play the role of \lq\lq UWD approximants" of the systems of the family \eqref{4.hamilperzl}: these last ones
inherit from them spectral and dynamical properties which will be essential in the
proof of Theorem~\ref{4.teormejor}. This proof also requires the characterization
of the exponential dichotomy and the uniform weak disconjugacy
in terms of the variation of the rotation number, recalled in properties \hyperlink{p7}{\bf p7}, \hyperlink{p8}{\bf p8} and \hyperlink{p9}{\bf p9},
which in turn
requires the existence of a $\sigma$-ergodic measure with full topological
support.
\begin{teor}\label{4.teormejor}
Suppose that Hypothesis~\ref{4.hipos} holds,
and that there exists a $\sigma$-ergodic
measure $m_0$ on $\W$ with full topological support. Let us define
\begin{equation}\label{4.defI}
 \mI:=\{\alpha\in\R\,|\;\eqref{4.hamilperzl}^\alpha \text{ has
 ED and satisfies NC\/}\}\,.
\end{equation}
and $\mI_0$ by \eqref{4.defI0}.
Then,
\begin{itemize}
\item[\rm(i)] there exists $\alpha^*\in(0,\infty]$ such that
$\mI=(-\infty,\alpha^*)$, and
\[
M^+(\w,\alpha_1)<M^+(\w,\alpha_2)
\]
for every
$\w\in\W$ and for every pair of elements $\alpha_1<\alpha_2$ of $\mI$.
\item[\rm(ii)] There exists a nonincreasing and continuous extended-real
function $\rho\colon\mI\to(0,\infty]$ such that \eqref{4.hamilperzle}$_\ep^\alpha$
has ED and is UWD for $\alpha\in\mI$ if and only if
$\ep\in(0,\rho(\alpha))$. In particular, for these values of $\ep$, there
exist the Weyl functions $M^\pm_\ep(\w,\alpha)$.
\item[\rm(iii)]
In addition, $\rho(\alpha)=\infty$ whenever $H_2-\alpha\Delta>0$, and
$\rho$ is strictly decreasing at the points at which it takes real values
(if they exist).
\item[\rm(iv)] If $\alpha^*<\infty$, then the family
\eqref{4.hamilperzl}$^{\alpha^*}$ does not have ED.
\end{itemize}
\end{teor}
An analysis of the presence of ED for the family \eqref{4.hamilperzl} assuming that
it satisfies the UWD property has been previously carried out in \cite{jnuo} and \cite{jonnf}. But Theorem~\ref{4.teormejor} improves this analysis significantly.
Its proof is based on the following result, which it extends, and which is part of
Theorem 7.31 of \cite{jonnf}. The result of \cite{jonnf} is formulated
in the case that $\W$ is minimal and assuming
that a certain condition D2 does not hold, but an identical proof
works for the statement we give now.
\begin{teor}\label{4.teorlibro}
Suppose that Hypothesis \ref{4.hipos} holds.
Let us define
\begin{equation}\label{4.defI0}
\begin{split}
 \mI_0:=\{0\}\cup\{\alpha_0\in\R\,|\;&\eqref{4.hamilperzl}^\alpha \text{ has
 ED and satisfies NC}\\
 &\text{for all }\alpha\in[0,\alpha_0)
 \text{ or } \alpha\in(\alpha_0,0]\}\}\,.
\end{split}
\end{equation}
Then,
\begin{itemize}
\item[\rm(i)] $\mI_0$ is an open interval containing $0$,
and $M^+(\w,\alpha_1)<M^+(\w,\alpha_2)$ for every $\w\in\W$ and for every
pair of elements $\alpha_1<\alpha_2$ of $\mI_0$.
\item[\rm(ii)] There exists a nonincreasing and lower semicontinuous extended-real
function $\rho\colon\mI_0\to(0,\infty]$ such that \eqref{4.hamilperzle}$_\ep^\alpha$
has ED and is UWD for $\alpha\in\mI_0$ if and only if
$\ep\in(0,\rho(\alpha))$.
\end{itemize}
\end{teor}
The fact that the set
$\mI$ defined by \eqref{4.defI} agrees with the set $\mI_0$
defined by \eqref{4.defI0} will be fundamental in the proof of Theorem~\ref{4.teormejor}.
\par
The proof of the main result is quite long. It requires two auxiliary results.
One of then, Theorem \ref{4.teorizq}, is new and has independent interest.
The other one, Theorem \ref{4.teor5.61}, can be found in \cite{jonnf}.
\par From
this point we divide this section in two parts. The first one is devoted to
formulate these auxiliary results and prove the new one. And the proof of Theorem \ref{4.teormejor} is given in the second one.
\subsection{Auxiliary results}\label{sec4.1}
We point out that the results of this section do not require the existence of a $\sigma$-ergodic measure on $\W$ with full support.
\par
The fist auxiliary result is part
of Theorem 5.61 of \cite{jonnf}.
\begin{teor}\label{4.teor5.61}
Suppose that $H_3\ge 0$ and $\Delta>0$, and suppose that the set
\[
\mI_1=\{\alpha\in\R\,|\;\eqref{4.hamilperzl}^\alpha \text{ is UWD}\}
\]
is nonempty. Then
$\mI_1=(-\infty,\wit\alpha]$ for a point $\wit\alpha\in\R$. In addition,
for $\alpha<\wit\alpha$, the family \eqref{4.hamilperzl}$^\alpha$ has exponential dichotomy over $\W$
and the Weyl functions exist and agree with the principal functions;
but it does not have exponential dichotomy for~$\wit\alpha$.
\end{teor}
In order to formulate the second auxiliary result, Theorem~\ref{4.teorizq},
we need to
introduce two new sets of families of linear Hamiltonian systems:
\begin{equation}\label{4.hamilperwl}
 \bw'=\left[\begin{array}{cc} -H_1^T(\wt)&
  H_2(\wt)-\lambda\Delta(\wt)\\H_3(\wt)&H_1(\wt)\end{array}\right]\bw
\end{equation}
for $\lambda\in\C$, and
\begin{equation}\label{4.hamilperwle}
 \bw'=\left[\begin{array}{cc} -H_1^T(\wt)&
  H_2(\wt)-\lambda\Delta(\wt)\\H_3(\wt)+\ep I_n&H_1(\wt)\end{array}\right]\bw
\end{equation}
for $\lambda\in\C$ and $\ep\in\R$.
We represent by
\eqref{4.hamilperwl}$^\lambda$ and \eqref{4.hamilperwle}$_\ep^\lambda$
the families
corresponding to a particular (real or complex) value of $\lambda$.
Again, \eqref{4.hamilperwle}$_0^\lambda$ agrees with
\eqref{4.hamilperwl}$^\lambda$.
In all the cases,
we will often substitute $\lambda$ by $\alpha$ if the parameter is real
\par
Before stating Theorem~\ref{4.teorizq}, we will derive
some facts from the relation of \eqref{4.hamilperwle}$_\ep^\lambda$ and
\eqref{4.hamilperwl}$^\lambda$
with \eqref{4.hamilperzle}$_\ep^\lambda$ and
\eqref{4.hamilperzl}$^\lambda$, and we will also establish some notation
which we will use below.
\begin{nota}\label{4.inver}
(a)~As in the proof of Corollary \ref{3.coro2}, it is easy to check that
the change of variables $\bw=\lsm 0_n&I_n\\I_n&0_n\rsm\bz$ takes
\eqref{4.hamilperzl}$^\lambda$ to \eqref{4.hamilperwl}$^\lambda$ and
\eqref{4.hamilperzle}$_\ep^\lambda$ to \eqref{4.hamilperwle}$_\ep^\lambda$;
hence, according to Remark~\ref{2.notasED}(b), the ED of
\eqref{4.hamilperzle}$_\ep^\lambda$ is equivalent to the ED of
\eqref{4.hamilperwle}$_\ep^\lambda$. Moreover, it follows
from \eqref{2.deflpm} that, in the case of ED,
$\lsm \bz_1\\\bz_2\rsm\in l^\pm_\ep(\w,\lambda)$
if and only if $\lsm\bz_2\\\bz_1\rsm\in\wit l^\pm_\ep(\w,\lambda)$,
where $l^\pm_\ep(\w,\lambda)$ and $\wit l^\pm_\ep(\w,\lambda)$ are
the Lagrange planes defined by~\eqref{2.defl} from
the families \eqref{4.hamilperzle}$_\ep^\lambda$ and
\eqref{4.hamilperwle}$_\ep^\lambda$.
And of course, the same happens with
\eqref{4.hamilperzl}$^\lambda$ and \eqref{4.hamilperwl}$^\lambda$.
\par
(b)~In the case of ED and of global
existence of the Weyl functions (just one or both of them)
for \eqref{4.hamilperzle}$_\ep^\lambda$ (or for \eqref{4.hamilperwle}$_\ep^\lambda$)
for $\lambda\in\C$ and $\ep\in\R$, we will represent them by
$M_\ep^\pm(\w,\lambda)$ (or by $\wit M_\ep^\pm(\w,\lambda)$). In particular,
$M_0^\pm(\w,\lambda)$ (or $\wit M_0^\pm(\w,\lambda)$) represent
the Weyl functions of \eqref{4.hamilperzl}$^\lambda$
(or of \eqref{4.hamilperwl}$^\lambda$), for which we will omit the
subscript: $M^\pm(\w,\lambda)$ (or $\wit M^\pm(\w,\lambda)$).
\par
(c)~Similarly, in the case of UWD of the family
\eqref{4.hamilperzle}$_\ep^\alpha$ (or of \eqref{4.hamilperwle}$_\ep^\alpha$)
for $\alpha\in\R$ and $\ep\in\R$, we will represent the corresponding principal
functions by
$N_\ep^\pm(\w,\alpha)$ (or by $\wit N_\ep^\pm(\w,\alpha)$); and we will denote
$N^\pm(\w,\alpha):=N_0^\pm(\w,\alpha)$ (and
$\wit N^\pm(\w,\alpha):=\wit N_0^\pm(\w,\alpha)$).
\par
(d)~Suppose that both $M_\ep^+(\w,\lambda)$ and $\wit M_\ep^+(\w,\lambda)$
exist. The relation between the solutions of
\eqref{4.hamilperzle}$_\ep^\lambda$ and \eqref{4.hamilperwle}$_\ep^\lambda$
show that $\bz=\lsm\bz_1\\\bz_2\rsm$ belongs to
the Lagrange plane given by $\lsm I_n\\ M_\ep^+(\w,\alpha)\rsm$ if and only if
$\bw=\lsm\bz_2\\\bz_1\rsm$ belongs to that given by
$\lsm M_\ep^+(\w,\alpha)\\I_n\rsm$, which agrees with that given by
$\lsm I_n\\ (M_\ep^+)^{-1}(\w,\alpha)\rsm$. This fact combined with
\eqref{2.deflpm} guarantees that
$\wit M_\ep^+(\w,\lambda)=(M_\ep^+)^{-1}(\w,\lambda)$. The same
property holds of course in the case of $M_\ep^-(\w,\lambda)$ and
$\wit M_\ep^-(\w,\lambda)$.
\par
(e)~For the same reason, suppose that \eqref{4.hamilperwle}$_\ep^\lambda$ has
ED, and that $\wit M_\ep^+(\w,\lambda)$ exists and
is nonsingular for all $\w\in\W$. Then $M_\ep^+(\w,\lambda)$ exists and
$M_\ep^+(\w,\lambda)=(\wit M_\ep^+)^{-1}(\w,\lambda)$. And the analogous
property holds in the case of $\wit M_\ep^-(\w,\lambda)$ and
$M_\ep^-(\w,\lambda)$.
\end{nota}
\begin{teor}\label{4.teorizq}
Suppose that Hypothesis \ref{4.hipos} holds, and let
$\alpha_0$ be a real value with $H_2-\alpha_0\Delta>0$.
Then, if $\alpha\le\alpha_0$, the family
\eqref{4.hamilperwl}$^\alpha$
is UWD and has ED, and the Weyl functions $\wit M^\pm(\w,\alpha)$
globally exist; and, in addition
\[
 \wit M^+(\w,\alpha_1)\le\wit  M^+(\w,\alpha_2)<0\le
 \wit M^-(\w,\alpha_2)\le\wit M^-(\w,\alpha_1)
\]
for every $\w\in\W$ if $\alpha_2<\alpha_1\le\alpha_0$.
\par
Consequently, if $\alpha\le\alpha_0$, the family
\eqref{4.hamilperzl}$^\alpha$
has ED, and the Weyl function $M^+(\w,\alpha)$
globally exists; and, in addition
\[
 M^+(\w,\alpha_2)\le M^+(\w,\alpha_1)<0
\]
for every $\w\in\W$ if $\alpha_2<\alpha_1\le\alpha_0$.
\end{teor}
\begin{proof}
Throughout the proof, we will use the notation established in
Remarks \ref{4.inver}(b)\&(c).
\par
We have $H_2-\alpha\Delta>0$ for $\alpha<\alpha_0$ and $H_3\ge 0$,
so that the family \eqref{4.hamilperwl}$^\alpha$ is UWD for
$\alpha<\alpha_0$: see property \hyperlink{p1}{\bf p1} of Section \ref{sec2}.
In addition $H_2-\alpha\Delta$ decreases as $\alpha$ increases, so that,
according to property \hyperlink{p5}{\bf p5},
the corresponding principal functions
$\wit N^\pm(\w,\alpha)$ satisfy
\begin{equation}\label{4.ordenN}
\wit N^+(\w,\alpha_1)\le \wit N^+(\w,\alpha_2)\le
\wit N^-(\w,\alpha_2)\le \wit N^-(\w,\alpha_1)
\quad\text{if $\alpha_2<\alpha_1\le\alpha_0$}\,.
\end{equation}
Let us check that the proof of
the theorem will be completed once we have proved that
the family \eqref{4.hamilperwl}$^\alpha$ satisfies
\[
 \wit N^+(\w,\alpha)<0\le \wit N^-(\w,\alpha)
 \quad\text{if $\alpha<\alpha_0$ for all $\w\in\W$}\,.
\]
If this is the case,
property \hyperlink{p3}{\bf p3} guarantees that
the family \eqref{4.hamilperwl}$^\alpha$
has ED and admits Weyl functions
with $\wit M^\pm(\w,\alpha)=\wit N^\pm(\w,\alpha)$. Therefore,
$\wit M^+(\w,\alpha)<0\le \wit M^-(\w,\alpha)$ for all $\w\in\W$
if $\alpha<\alpha_0$; and these facts,
\eqref{4.ordenN}, and Remarks~\ref{4.inver}(a)\&(e) prove all the assertions
of the theorem. (Note that $M_1\le M_2<0$ ensure that there exist
$M_1^{-1}$ and $M_2^{-1}$ with $M_2^{-1}\le M_1^{-1}$).
\par From this point, and for the safe of clarity, we divide
the proof in three steps.
\smallskip\par
{\bf First step.} We will start by proving the next statements:
\begin{itemize}
\item[\hypertarget{1}\text{\bf s1.}] For any $\lambda\in\C-\R$
there globally exist the Weyl functions $\wit M^\pm(\w,\lambda)$.
In addition, the maps
$M^\pm\colon\W\times(\C-\R)\to\s_n(\C),\;(\w,\lambda)\mapsto M^\pm(\w,\lambda)$
are jointly continuous, and they
are holomorphic on $\C-\R$ for each $\w\in\W$ fixed.
In particular, they are Herglotz functions on $\C^+$ and $\C^-$.
\item[\hypertarget{2}\text{\bf s2.}] Moreover,
$\mp\Rea\wit M^\pm(\w,\lambda)\ge 0$ and
$\mp\Ima\wit M^\pm(\w,\lambda)\ge 0$
whenever $\Rea\lambda\le\alpha_0$ and $\Ima\lambda>0$.
\end{itemize}
\par
Let us take $\ep>0$ and $\alpha\in(-\infty,\alpha_0]$.
Since $H_2-\alpha\Delta>0$ and $H_3+\ep I_n>0$,
Lemma \ref{2.lemaRic}(i) ensures that
the family \eqref{4.hamilperzle}$^\alpha_\ep$ has ED and that
the corresponding (real) Weyl functions $M^\pm_\ep(\w,\alpha)$
exist and satisfy
\begin{equation}\label{4.a1}
 \mp M_\ep^\pm(\w,\alpha)>0 \quad
 \text{for all $\,\w\in\W\,$ if $\,\ep>0\,$ and $\,\alpha\le\alpha_0$}\,.
\end{equation}
Therefore, if $\ep>0$ and $\alpha\le\alpha_0$, then
the family \eqref{4.hamilperwle}$_\ep^\alpha$
has ED and there exist the
Weyl functions $\wit M_\ep^\pm(\w,\alpha)$ for
\eqref{4.hamilperzl}$_\ep^\alpha$, with $\mp\wit M_\ep^\pm(\w,\alpha)=
\mp(M_\ep^\pm)^{-1}(\w,\alpha)$: see Remarks~\ref{4.inver}(a)\&(e).
In particular, it follows from
\eqref{4.a1} that
\begin{equation}\label{4.a2}
 \mp\wit M_\ep^\pm(\w,\alpha)>0 \quad
 \text{for all $\,\w\in\W\,$ if $\,\ep>0\,$ and $\,\alpha\le\alpha_0$}\,.
\end{equation}
\par
Hypothesis~\ref{4.hipos} and Remark~\ref{4.inver}(a)
ensure that the family \eqref{4.hamilperwle}$_0^0$ has
ED. Hence there exists $\ep_0$
such that \eqref{4.hamilperwle}$_\ep^0$ has ED for
$\ep\in[0,\ep_0)$: see Theorem \ref{2.teorSSPT}.
We deduce from Theorem~\ref{3.teornoatki}(i)\&(ii) the next properties:
if $\ep\in[0,\ep_0)$, $\alpha\in\R$ and $\beta>0$,
then the family \eqref{4.hamilperwle}$_\ep^{\alpha+i\beta}$
has ED and there globally exist the Weyl
functions $\wit M^\pm_\ep(\w,\alpha+i\beta)$.
Note that $\wit M^\pm_0(\w,\alpha+i\beta)=\wit M^\pm(\w,\alpha+i\beta)$.
The information provided in Theorem~\ref{3.teornoatki}(ii) concerning
continuity and analyticity completes the proof of property \hyperlink{1}{\bf s1}.
In addition, also according to Theorem~\ref{3.teornoatki}(ii),
\begin{equation}\label{4.a3}
\mp\Ima\wit M^\pm_\ep(\w,\alpha+i\beta)\ge 0
\quad\text{if $\ep\in[0,\ep_0)$, $\alpha\in\R$ and $\beta>0$}\,.
\end{equation}
Moreover, we can ensure that
\begin{equation}\label{4.a5}
 \mp\Rea \wit M^\pm_\ep(\w,\alpha+i\beta)>0\quad
 \text{for all $\w\in\W\,$
 if $\ep\in(0,\ep_0)$, $\alpha\le\alpha_0$ and $\beta>0$}\,.
\end{equation}
In order to prove this last assertion,
we fix $\ep\in(0,\ep_0)$ and $\alpha\le\alpha_0$,
and use Theorem \ref{2.teorSSPT}
to ensure that
\begin{equation}\label{4.a4}
 \lim_{\beta\to 0^+}\wit M^\pm_\ep(\w,\alpha+i\beta)=
 \wit M^\pm_\ep(\w,\alpha) \quad\text{uniformly~in $\,\w\in\W$}\,.
\end{equation}
Therefore, it follows from
\eqref{4.a2} that there exists $\beta_0=\beta(\ep,\alpha)$ such that
\[
 \mp\Rea\wit M^\pm_\ep(\w,\alpha+i\beta)>0
 \quad \text{for all $\,\w\in\W\,$ if $\,\beta\in[0,\beta_0)$}\,.
\]
Let us work now with $\wit M^-$, assuming for contradiction that the
value
\[
 \beta^*=\beta^*(\ep,\alpha):=\sup\{\beta_0>0\,|\;
 \Rea\wit M^-_\ep(\w,\alpha+i\beta)>0\;\text{for all $\w\in\W$
 if $\beta\in[0,\beta_0)$}\}
\]
is finite. We denote $R^\pm(\w):=\Rea \wit M_\ep^\pm(\w,\alpha+i\beta^*)$
and $I^\pm(\w):=\Ima \wit M_\ep^\pm(\w,\alpha+i\beta^*)$. A straightforward
computation from the Riccati equation (see \eqref{2.riccati}) associated to
\eqref{4.hamilperwle}$_\ep^{\alpha+i\beta^*}$ shows that, for all $\w\in\W$,
the maps $t\mapsto R^\pm(\wt)$ are solutions of the
Riccati equation associated to the (real) Hamiltonian system
\[
 \bu'=\left[\begin{array}{cc} -H_1^T+\beta\Delta I^\pm&
  H_2-\alpha\Delta\\
  I^\pm(H_2-\alpha\Delta)\,I^\pm+H_3+\ep I_n&
  H_1-\beta I^\pm\Delta\end{array}\right]\bu
\]
(where $H_1,\,H_2,\,H_3,\,\Delta$ and $I^\pm$ have argument $\wt$).
We know also that they are globally defined and that $R^-(\w)\ge 0$.
Since $H_2-\alpha\Delta>0$ and $H_2-\alpha\Delta+H_3+\ep I_n>0$,
Lemma \ref{2.lemaRic}(ii) shows that $R^-(\w)>0$, which contradicts
the definition of $\beta^*$.
So, \eqref{4.a5} is proved for $\wit M^-$, and the proof for $\wit M^+$ is
analogous.
\par
Finally, we can also deduce from Theorem \ref{2.teorSSPT} that
\begin{equation}\label{4.a6}
\begin{split}
 &\lim_{\ep\to 0^+}\wit M^\pm_\ep(\w,\lambda)=
 \wit M^\pm(\w,\lambda)\quad\text{uniformly}\\
 &\qquad\quad\text{in the compact sets
 of $\,\W\times\C^+\,$ and $\,\W\times\C^-$}\,,
\end{split}
\end{equation}
which together with \eqref{4.a3} and \eqref{4.a5} proves
property \hyperlink{2}{\bf s2}. This completes the first step.
\smallskip\par
{\bf Second step.} We fix $\w\in\W$ and prove the next assertions.
\begin{itemize}
\item[\hypertarget{3}\text{\bf s3.}]
Let $\wit\alpha_0>0$ satisfy
$H_2-\wit\alpha_0\Delta>0$. Then
$\lim_{\beta\to 0^+} \Ima\wit M^-(\w,\alpha+i\beta)=0_n$
uniformly on the compact subsets of $(-\infty,\wit\alpha_0)$.
\item[\hypertarget{4}\text{\bf s4.}] In addition, there exist the limits
$\wit F^\pm(\w,\alpha):=\lim_{\beta\to 0^+}\wit M^\pm(\w,\alpha+i\beta)$
for all $\alpha\le\alpha_0$, and they are real matrices
with $\mp \wit F^\pm(\w,\alpha)\ge 0$. Moreover,
the matrix-valued functions $t\mapsto \wit F^\pm(\wt,\alpha)$
are two globally defined solutions of the Riccati
equation associated to \eqref{4.hamilperwl}$^\alpha$.
\end{itemize}
\par
Note that, in the first step,
nothing precludes us from substituting $\alpha_0$ by a slightly
greater $\wit\alpha_0$ for which $H_2-\wit\alpha_0\Delta>0$:
Properties \hyperlink{2}{\bf s2\,} and \eqref{4.a1}
are true for this $\wit\alpha_0$. We will use both of them.
Let $\ep_0$ be the real number defined
at the beginning of the first step.
As seen in \eqref{4.a3}, the holomorphic maps $\C^+\to\s_n(\C),\;
\lambda\mapsto \wit M_\ep^-(\w,\lambda)$
are Herglotz for $\ep\in[0,\ep_0)$ (see Definition~\ref{3.defherg}). Therefore,
Theorem~\ref{4.teorherg}(ii) (see below) provides the representation
\[
 \wit M_\ep^-(\w,\lambda)=L_\ep+K_\ep\,\lambda+
 \int_\R\left(\frac{1}{t-\lambda}-\frac{t}{t^2+1}\right)dP_\ep(t)
\]
for $\Ima\lambda>0$ and $\ep\in[0,\ep_0)$. In the case $\ep=0$
we rewrite this as
\[
 \wit M^-(\w,\lambda)=L+K\,\lambda+\int_\R\left(\frac{1}{t-\lambda}-
 \frac{t}{t^2+1}\right)dP(t)\,.
\]
Let us take $\ep\in(0,\ep_0)$.
Theorem~\ref{4.teorherg}(iii) can be combined with the property
\eqref{4.a4} and with the real character of
$\wit M_\ep^-(\w,\alpha)$ for $\alpha\le\wit\alpha_0$
(see \eqref{4.a1}) in order to see that
\[
 \frac{1}{2}\,(P_\ep\{\alpha_1\}
 +P_\ep\{\alpha_2\})+\int_{(\alpha_1,\alpha_2)}dP_\ep(t)=
 \frac{1}{\pi}\,\int_{\alpha_1}^{\alpha_2}
 \Ima\wit M_\ep^-(\wD,\alpha)\,d\alpha=0_n
\]
whenever $\alpha_1<\alpha_2\le\wit\alpha_0$. This ensures that
\begin{equation}\label{4.b3}
 \int_{(-\infty,\wit\alpha_0)} dP_\ep(t)=0_n \quad\text{if $\ep\in(0,\ep_0)$}\,.
\end{equation}
Now we take a sequence $(\ep_m)\downarrow 0$ and recall that
$\lim_{m\to\infty} \wit M_{\ep_m}^-(\w,\lambda)=\wit M^-(\w,\lambda)$
uniformly on the compact sets of $\C^+$ (see \eqref{4.a6}). Therefore
the sequence $(dP_{\ep_m})$ converges to $dP$ in the weak$^*$ sense
(see Theorem \ref{4.teorP} below), which together with \eqref{4.b3}
allows us to check that
\begin{equation}\label{4.b44}
 \int_{(-\infty,\wit\alpha_0)}dP(t)=0_n\,.
\end{equation}
Let us take $\delta>0$ and $\alpha_1<\wit\alpha_0-\delta$
and denote $\mC^\delta_{\alpha_1}:=
\{\lambda\in\C\,|\;\Rea\lambda\in[\alpha_1,\wit\alpha_0-\delta]\linebreak
\text{and }\Ima\lambda\in(0,1]\}$.
Note that there exist constants $c_1>0$ and $c_2>c_1$ such that
\begin{equation}\label{4.b45}
 c_1\le\Frac{|t-\lambda|^2}{t^2+1}\le c_2
 \quad\text{for $\;t\in[\wit\alpha_0,\infty)\;$ and
 $\;\lambda\in\mC^\delta_{\alpha_1}$}\,,
\end{equation}
since the function is continuous, takes strictly
positive values (due to $|t-\lambda|^2\ge \delta^2$),
and $\lim_{t\to\infty}|t-\lambda|^2/(t^2+1)=1$.
Then, if $\lambda=\alpha+i\beta\in \mC^\delta_{\alpha_1}$,
\[
\begin{split}
 &\frac{1}{\beta}\,\Ima\wit M^-(\w,\alpha+i\beta)=
 K+\int_\R\frac{1}{|t-\lambda|^2}\:dP(t)\\
 &\qquad\qquad
 =K+\int_{[\alpha_0,\infty)}\frac{1}{|t-\lambda|^2}\:dP(t)
 \le K+\frac{1}{c_1}\int_{[\alpha_0,\infty)}\frac{1}{t^2+1}\:dP(t)\\
 &\qquad\qquad
 \le K+\frac{1}{c_1}\int_{\R}\frac{1}{t^2+1}\:dP(t)\le
 K+\frac{1}{c_1}\,\Ima\wit M^-(\w,i)\,.
\end{split}
\]
Here we have used Theorem \eqref{4.teorherg}(ii) at the
first and last steps, and \eqref{4.b44} and \eqref{4.b45} at the
second and third steps. This and property \hyperlink{2}{\bf s2} yield
\[
 0_n\le
 \Ima\wit M^-(\w,\alpha+i\beta)\le\beta\left(K+\frac{1}{c_1}\,\Ima\wit M^-(\w,i)
 \right)
\]
if $\alpha\in[\alpha_1,\wit\alpha_0-\delta]$ and $\beta\in(0,1]$.
Property \hyperlink{3}{\bf s3} follows easily from here.
\par
In turn, property \hyperlink{3}{\bf s3\,} guarantees that
whenever a sequence $(\lambda_m)$ in $\C^+$ converges to
$\alpha\in(-\infty,\wit\alpha_0)$, it is
$\lim_{m\to\infty} \Ima\wit M^-(\w,\lambda_m)=0_n$. This fact
allows us to apply the Schwarz
reflection principle (see e.g.~\cite{rudi}, Theorem 11.14) in order
to ensure that $\wit M^-(\w,\lambda)$ admits a holomorphic extension
to $\C-(\alpha_0,\infty)$ (which is contained in
$\C-[\wit\alpha_0,\infty)$) with null imaginary part
for $\lambda=\alpha\in(-\infty,\alpha_0]$.
We call $\wit F^-(\w,\alpha)$ to the restriction of
this extension to $(-\infty,\alpha_0]$. Note that this process
can be performed for all $\w\in\W$.
In particular,
\[
 \wit F(\wt,\alpha)=\lim_{\beta\to 0^+}\wit M^-(\wt,\alpha+i\beta)
\]
for all $t\in\R$. Since $t\mapsto\wit M^-(\wt,\alpha+i\beta)$ solves the
Riccati equation associated to \eqref{4.hamilperwl}$^{\alpha+i\beta}$,
we conclude that $t\mapsto \wit F^-(\wt,\alpha)$ solves the Riccati
equation associated to \eqref{4.hamilperwl}$^\alpha$.
Finally, it follows from
\hyperlink{2}{\bf s2\,} that $\wit F^-(\w,\alpha)\ge 0$. Hence,
property \hyperlink{4}{\bf s4\,} is proved.
\par
The proofs of \hyperlink{3}{\bf s3\,}\ and \hyperlink{4}{\bf s4\,}
are analogous in the case of $\wit M^+$: the second step is complete.
\smallskip\par
{\bf Third step.} By combining the previous property \hyperlink{4}{\bf s4\,} with
property \hyperlink{p4}{\bf p4} of Section~\ref{sec2}, we obtain
\begin{equation}\label{4.c1}
 \!\!\!\!
 \wit N^+(\w,\alpha)\le \wit F^+(\w,\alpha)\le 0\le \wit F^-(\w,\alpha)
 \le \wit N^-(\w,\alpha)
 \quad\text{if $\alpha\in(-\infty,\alpha_0]$}
\end{equation}
for all $\w\in\W$, where $\wit N^\pm(\w,\alpha)$ are the principal functions for
\eqref{4.hamilperwl}$^\alpha$, whose existence has been
guaranteed at the beginning of the proof.
This fact will allow us to prove the next assertion.
\begin{itemize}
\item[\hypertarget{5}\text{\bf s5.}]
 $\wit N^+(\w,\alpha)<0\le \wit N^-(\w,\alpha)$ for all $\w\in\W$ if
 $\alpha<\alpha_0$.
\end{itemize}
Note that, as explained before the first step, this property completes the proof.
Note also that the second inequality is already proved: see \eqref{4.c1}.
\par
So we must just prove that $\wit N^+(\w,\alpha)<0$.
We will use below this immediate consequence of
Theorem~\ref{4.teorlibro}(i):
\begin{equation}\label{4.existealfa}
\text{there exists $\alpha_1>0$ such that
$[-\alpha_1,\alpha_1]\subset \mI$}\,.
\end{equation}
Let us proceed by contradiction, assuming that there exists
$\wit\alpha\le\alpha_0$ such that \lq\lq$\wit N^+(\w,\wit\alpha)<0$" is false.
Then, since
\eqref{4.ordenN} holds, there exists $\wit\w\in\W$
and a vector $\bv\in\R^n$, $\bv\ne\bcero$
(which we fix from now on),
such that $\bv^T\wit N^+(\wit\w,\wit\alpha)\,\bv=0$.
It also follows from \eqref{4.ordenN} and \eqref{4.c1} that
\begin{equation}\label{4.c3}
 \bv^T\wit N^+(\wit\w,\alpha)\,\bv=0 \quad\text{if $\alpha\le\wit\alpha$}\,.
\end{equation}
Let us define the holomorphic function $h\colon\C^+\to\C$ by
\[
 h(\alpha+i\beta):=\bv^T\wit M^+(\wit\w,\alpha+i\beta)\,\bv\,.
\]
As seen in the proof of property \hyperlink{3}{\bf s3},
there exists a holomorphic extension $\wit h$ of $h$ to
$\C-[\wit\alpha,\infty)$. Note also that
\eqref{4.c1} and \eqref{4.c3} yield
\[
 \wit h(\alpha)=\lim_{\beta\to 0^+} h(\alpha+i\beta)=
 \bv^T\wit F^+(\wit\w,\alpha)\,\bv=0
 \quad\text{for all $\alpha\le\wit\alpha$}\,.
\]
The principle of isolated zeroes ensures then that
\[
 \wit h(\alpha+i\beta)=\bv^T\wit M^+(\wit\w,\alpha+i\beta)\,\bv=0
 \quad\text{if $\alpha\in\R$ and $\beta>0$}\,.
\]
Therefore
$\bv^T\Rea \wit M^+(\wit\w,\alpha+i\beta)\,\bv=0$ and
$\bv^T\Ima \wit M^+(\wit\w,\alpha+i\beta)\,\bv=0$ for
$\alpha\in\R$ and $\beta>0$.
These equalities combined with property \hyperlink{2}{\bf s2\,} ensure that,
if $\alpha\le\wit\alpha$ and $\beta>0$, then
$\Rea\wit M^+(\wit\w,\alpha+i\beta)\,\bv=\bcero$ and
$\Ima\wit M^+(\wit\w,\alpha+i\beta)\,\bv=\bcero$.
In other words, $\wit M^+(\wit\w,\alpha+i\beta)\,\bv=\bcero$
if $\alpha\le\wit\alpha$ and $\beta>0$. A new application of the
principle of isolated zeroes, now to components of
the holomorphic vector function
$\wit M^+(\wit\w,\alpha+i\beta)\,\bv$ defined on the upper complex
half-plane, shows that $\wit M^+(\wit\w,\alpha+i\beta)\,\bv=\bcero$
if $\alpha\in\R$ and $\beta>0$. Now we use property \hyperlink{3}{\bf s3\,}
in order to deduce from Theorem~\ref{4.teorherg}(i)
the existence of a point $\alpha_2\in[-\alpha_1,\alpha_1]$
(where $\alpha_1$ satisfies \eqref{4.existealfa})
such that there exists
\begin{equation}\label{4.c5}
 \wit F^+(\wit\w,\alpha_2):=
 \lim_{\beta\to 0^+} \wit M^+(\wit\w,\alpha_2+i\beta)\,,
\end{equation}
so that $\wit F^+(\wit\w,\alpha_2)\,\bv=\bcero$.
Theorem \ref{2.teorSSPT}
combined with \eqref{4.c5} ensure that the Lagrange plane
$\wit l^+(\wit\w,\alpha_2)$ of
\eqref{4.hamilperwl}$^{\alpha_2}$ (see \eqref{2.defl})
can be represented by $\lsm I_n\\\wit F^+(\wit\w,\alpha_2)\rsm$. And we
already know that it can be also represented by
$\lsm M^+(\wit\w,\alpha_2)\\I_n\rsm$: see Remark~\ref{4.inver}(a).
This means that $\wit F^+(\wit\w,\alpha_2)$ is a nonsingular matrix (it agrees with
$(M^+)^{-1}(\wit\w,\alpha_2)$), which contradicts the equality
$\wit F^+(\wit\w,\alpha_2)\,\bv=\bcero$.
We have arrived to the sought-for contradiction, and
hence the proof is complete.
\end{proof}
We complete Section \ref{sec4.1} by formulating
the results on Herglotz matrix-valued functions which we have used
in the proof of Theorem~\ref{4.teorizq}.
The first one can be found in~\cite{koos} and~\cite{gets}, and a proof of the second
one is given in Theorem 3.15 of~\cite{jonnf}.
\begin{teor}\label{4.teorherg}
Let $G\colon\C^+\to\s_n(\C)$ be a Herglotz function, with $\Ima G\ge 0$. Then,
\begin{itemize}
\item[\rm(i)] for Lebesgue a.e.~$\alpha\in\R$
there exists the nontangential limit from the upper half-plane
$\lim_{\lambda\searrow\alpha} G(\lambda)$.
\item[\rm(ii)] There exist real symmetric matrices $L$ and $K$
and a real matrix-valued function $P(t)$
defined for $t\in\R$, which is symmetric, nondecreasing and right-continuous,
such that the {\em Nevalinna--Riesz--Herglotz representation\/}
\begin{equation}\label{4.nrh}
 G(\lambda)=L+K\,\lambda+\int_\R\left(\frac{1}{t-\lambda}-
 \frac{t}{t^2+1}\right)dP(t)
\end{equation}
holds for $\lambda\in\C^+$, with
\[
 L=\Rea G(i)\qquad\text{and}\qquad
 K=\lim_{\beta\to\infty}\frac{1}{i\beta}\:G(i\beta)\ge 0\,.
\]
\item[\rm(iii)] Let us represent $P\{\alpha\}=P(\alpha^+)-
P(\alpha^-)=
P(\alpha)-\lim_{\mu\to\alpha^-}P(\mu)$ for $\alpha\in\R$.
The {\em Stieltjes inversion formula}
\[
 \frac{1}{2}\,(P\{\alpha_1\}
 +P\{\alpha_2\})+\int_{(\alpha_1,\alpha_2)}dP(t)=
 \frac{1}{\pi}\,\lim_{\beta\to 0^+}\int_{\alpha_1}^{\alpha_2}
 \Ima G(\alpha+i\beta)\,d\alpha
\]
holds. In addition,
\[
\begin{split}
 P\{\alpha\}&
 =\lim_{\beta\to 0^+} \beta\,\Ima G(\alpha+i\beta)=
 -i\lim_{\beta\to 0^+}\beta\, G(\alpha+i\beta)\,,\\
 0&=\lim_{\beta\to 0^+} \beta\,\Rea G(\alpha+i\beta)\,.
\end{split}
\]
In particular, the matrix-valued measure $dP$ in representation \eqref{4.nrh} is uniquely
determined.
\end{itemize}
\end{teor}
\begin{teor}\label{4.teorP}
Let $(G_m)$ (for $m\in\N$) and $G_*$ be symmetric Herglotz
matrix-valued functions defined on $\C^+$ and with positive
semidefinite imaginary parts. Suppose that
$G_*(\lambda)=\lim_{m\to\infty}G_m(\lambda)$ uniformly on the
compact subsets of $\C^+$, and write
\[
\begin{split}
 G_m(\lambda)&=L_m+K_m\,\lambda+\int_\R\left(\frac{1}{t-\lambda}-
 \frac{t}{t^2+1}\right)
 dP_m(t)\,,\\
 G_*(\lambda)&=L_*+K_*\,\lambda+\int_\R\left(\frac{1}{t-\lambda}-
 \frac{t}{t^2+1}\right)
 dP_*(t)\,.
\end{split}
\]
Then, $dP_*=\lim_{m\to\infty}dP_m$ in the weak$^*$ sense; that
is,
\[
 \lim_{m\to\infty}\int_\R f^*(t)\,dP_m(t)\,
 f(t)=\int_\R f^*(t)\,dP_*(t)\,f(t)
\]
for every $f\colon\R\to\C^{\,2n}$ continuous and with compact support.
\end{teor}
\subsection{Proof of the main result}
Now we can finally prove Theorem~\ref{4.teormejor}.
\medskip\par\noindent
{\em Proof of Theorem~\ref{4.teormejor}}.
The notation established in Remark~\ref{4.inver} will be used in this proof.
\smallskip\par
(i) It is obvious that $\mI_0\subseteq\mI$, where $\mI_0$ and $\mI$ are respectively
defined by \eqref{4.defI0} and \eqref{4.defI}.
We will prove that
\begin{equation}\label{4.goal}
(-\infty,0]\subset\mI\,.
\end{equation}
Let us first explain why this proves (i). If there exists $\alpha_*\in I$ with
$\alpha_*>0$, we can replace \eqref{4.hamil} by \eqref{4.hamilperwl}$^{\alpha_*}$
in Hypothesis~\ref{4.hipos} in order to conclude that $(-\infty,\alpha_*]\subset\mI$.
This ensures that $\mI$ is either a negative half-line (containing $(-\infty,0]$)
or the whole $\R$. It follows
trivially that $\mI$ agrees with $\mI_0$, so that Theorem~\ref{4.teorlibro}(i)
completes the proof of (i).
\par
So, proving \eqref{4.goal} is our goal.
Let us take $\alpha_0\in\R$ as in the statement of Theorem \ref{4.teorizq},
so that the systems \eqref{4.hamilperzl}$^\alpha$
with $\alpha\in(-\infty,\alpha_0]$ have ED and satisfy NC. This ensures that
$(-\infty,\alpha_0]\subseteq\mI$.
We assume that $\alpha_0<0$ (otherwise there is nothing to prove).
Note that we must just prove that the family \eqref{4.hamilperzl}$^{\alpha}$ has ED
and satisfies NC for $\alpha\in(\alpha_0,0)$.
\par
We will first prove the assertion concerning ED.
Let us fix $\alpha_1\in(\alpha_0,0)$.
The robustness of the ED and NC (see Theorem \ref{2.teorSSPT}) and
Theorem~\ref{4.teorlibro}(ii) allow us to choose
$\ep_0<0$ close enough to $0$ as to guarantee these two conditions:
the families \eqref{4.hamilperzle}$_{\ep}^{\alpha_0}$ have ED and satisfy NC
for all $\ep\in[\ep_0,0]$; and if
the point $(0,\ep_2)$ belongs to the line $\mathcal{R}$ (in the
$(\alpha,\ep)-$plane) which joins $(\alpha_0,\ep_0)$ with
$(\alpha_1,0)$ (so that $\ep_2>0$), then
the family \eqref{4.hamilperzle}$_{\ep_2}^{0}$ has ED and is
UWD (or, using the words of Theorem~\ref{4.teorlibro}(ii),
$\ep_2\in(0,\rho(0))$). Note that the points of $\mathcal{R}$ are
$(\alpha_1+\gamma, \gamma\,(\ep_2/(-\alpha_1)))$
for $\gamma:=\alpha-\alpha_1$, and that the
families \eqref{4.hamilperzle}$_{\ep}^{\alpha}$ for points
$(\alpha,\ep)\in\mathcal{R}$ can be rewritten as
\[
 \bz'=\left[\begin{array}{cc} H_1(\wt)&
 H_3(\wt)+\gamma(-\ep_2/\alpha_1)\,I_n\\
 H_2(\wt)-\alpha_1\Delta(\wt)-\gamma\Delta(\wt)
 &-H_1^T(\wt)\end{array}\right]\bz\,;
\]
that is, as
\begin{equation}\label{4.hamilperg}
 \bz'=\big(H^{\alpha_1}(\wt)+\gamma\J\,\G(\wt)\big)\,\bz
\end{equation}
for $\Gamma:=\lsm \Delta&0_n\\0_n&(-\ep_2/\alpha_1)\,I_n \rsm$.
Since $\Gamma>0$, it satisfies the Atkinson condition
\eqref{3.atkgen} for all $\w_0\in\W$ (see Remark~\ref{3.notaatki}(e)). In
addition, this family has ED for the positive point
$\gamma_1:=-\alpha_1$
(since in this case we have the family
\eqref{4.hamilperzle}$_{\ep_2}^{0}$),
so that the families corresponding to values of $\gamma$ close
enough to $\gamma_1$ also have ED.
We deduce from property \hyperlink{p7}{\bf p7} of Section \ref{sec2} that
the rotation number with respect to $m_0$ is constant for all the families
\eqref{4.hamilperg} corresponding to the values of $\gamma$ in an open interval
centered in $\gamma_1$. As a matter of fact, it is 0 at
$\gamma_1$, since the family \eqref{4.hamilperzle}$_{\ep_2}^{0}$
satisfies $H_3+\ep_2I_n>0$ and is UWD: see property \hyperlink{p9}{\bf p9}.
The rotation number is also 0 at the negative point
$\gamma_0=\alpha_0-\alpha_1$. To check this assertion note that,
for this value of $\gamma$, the family agrees with
\eqref{4.hamilperzle}$_{\ep_0}^{\alpha_0}$, and that
the families \eqref{4.hamilperzle}$_{\ep}^{\alpha_0}$
have ED for all $\ep\in[\ep_0,0]$; deduce that
the rotation number with respect to $m_0$ is the same for all these families
(as property \hyperlink{p8}{\bf p8} ensures);
and note that the rotation number of the
family \eqref{4.hamilperzle}$_0^{\alpha_0}$, which satisfies NC
with $H_3\ge 0$, is 0 (see property \hyperlink{p10}{\bf p10}).
In addition, the rotation number increases as
$\gamma$ increases (see property \hyperlink{p6}{p6}).
Therefore, it is $0$ for $\gamma\in[\gamma_0,\gamma_1]$,
and this and the required condition $\Supp m_0=\W$
ensure that the families corresponding to $(\gamma_0,\gamma_1)$
have ED: see property \hyperlink{p8}{\bf p8}.
This includes the family corresponding to $\gamma=0$, which
is \eqref{4.hamilperzl}$^{\alpha_1}$. Our assertion concerning the ED is proved.
\par
Let us now prove that \eqref{4.hamilperzl}$^{\alpha}$ also
satisfies the NC (i.e., that $M^+(\w,\alpha)$ globally exists)
for $\alpha\in[\alpha_0,0]$,
which will complete the proof of \eqref{4.goal} and hence of (i).
We define
\[
 \wit\mI:=\{\wit\alpha\in[\alpha_0,0]\,|\;
 \text{\eqref{4.hamilperzl}$^\alpha$ has ED and satisfies NC for
 all $\alpha\in(-\infty,\wit \alpha]$}\}\,.
\]
Note that $\wit\mI$ is a nonempty and open subset of $[\alpha_0,0]$,
since $(-\infty,\alpha_0]\subset\mI$ (see Theorem \ref{2.teorSSPT});
and hence that $\alpha_2:=\sup\wit\mI>\alpha_0$.
The goal is to prove that \eqref{4.hamilperzl}$^{\alpha_2}$
satisfies the NC: if so, $\alpha_2\in\wit\mI$, which ensures that
$\alpha_2=0$ and hence that $(-\infty,0]\subset\mI$.
We take a strictly increasing
sequence $(\wit\alpha_m)$ in $[\alpha_0,\alpha_2)$ with limit $\alpha_2$,
and take $\ep_0>0$ such that, if $\ep\in[0,\ep_0]$, then: the families
\eqref{4.hamilperzle}$_\ep^\alpha$ have ED for $\alpha\in[\wit\alpha_1,0]$;
and there globally exist $M_\ep^+(\w,0)$.
Since $H_3+\ep_0 I_n>0$, property \hyperlink{p2}{\bf p2} ensures
that the family \eqref{4.hamilperzle}$_{\ep_0}^0$ is UWD. And since
$H_3+\ep_0 I_n>0$ increases as $\ep$ increases and $H_2-\alpha\Delta$
decreases as $\alpha$ increases, property \hyperlink{p5}{\bf p5}
guarantees that all the families
\eqref{4.hamilperzle}$_\ep^{\alpha}$ for $\alpha\in[\wit\alpha_1,0]$
and $\ep\in(0,\ep_0]$
are also UWD, with
\[
 N_{\ep}^+(\w,\wit\alpha_m)\le N_{\ep}^+(\w,\wit\alpha_{m+1})\le
 N_\ep^+(\w,\alpha_2)\le N_{\ep_0}^+(\w,\alpha_2)
\]
for all $\w\in\W$. In addition, property \hyperlink{p3}{\bf p3} ensures that
$N_\ep^+(\w,\alpha)=M_\ep^+(\w,\alpha)$, so that
\[
 M_{\ep}^+(\w,\wit\alpha_m)\le M_{\ep}^+(\w,\wit\alpha_{m+1})\le
 M_\ep^+(\w,\alpha_2)\le M_{\ep_0}^+(\w,\alpha_2)
\]
for all $\w\in\W$\,.
On the other hand, $M^+(\w,\alpha)=\lim_{\ep\to 0^+}M_{\ep}^+(\w,\alpha)$
for all $\alpha\in[\wit\lambda_1,\lambda_3)$, as
Theorem~\ref{2.teorSSPT}
ensures; so that
\[
 M^+(\w,\wit\alpha_m)\le M^+(\w,\wit\alpha_{m+1})\le
 M_{\ep_0}^+(\w,\alpha_2)\,.
\]
Therefore, there exists $F^+(\w,\alpha_2):=
\lim_{m\to\infty}M^+(\w,\wit\alpha_m)$ for all $\w\in\W$, which ensures
that the Lagrange plane represented by $\lsm I_n\\F^+(\w,\alpha_2)\rsm$
is the limit in the Lagrangian manifold
of those given by $\lsm I_n\\M^+(\w,\wit\alpha_m)\rsm$
(see e.g.~Proposition 1.25 of~\cite{jonnf}); that is, of
the sequence ($l^+(\w,\wit\alpha_m)$).
Theorem \ref{2.teorSSPT}
ensures that $\lsm I_n\\F^+(\w,\alpha_2)\rsm$ represents
$l^+(\w,\alpha_2)$, so that $M^+(\w,\alpha_2)$ globally exists
(and agrees with $F^+(\w,\alpha_2)$).
This completes the proof of (i).
\smallskip\par
(ii) We already know that $\mI=\mI_0$. Therefore, Theorem~\ref{4.teorlibro}(ii)
ensures that the function $\rho$ takes
posirive values; and that it
is lower semicontinuous and nonincreasing, so that it is continuous
from the right. Let us
take $\alpha_0\in\mI$ and a sequence $(\alpha_m)\uparrow\alpha_0$, and call
$\rho_0:=\lim_{m\to\infty}\rho(\alpha_m)$ (with $\rho_0\le\infty$).
We know that $\rho(\alpha_0)\le\rho_0$, and our goal is to prove
that they are equal. Or, in other words, that the family
\eqref{4.hamilperzle}$^\ep_{\alpha_0}$ has ED and is UWD for
$\ep\in(0,\rho_0)$.
\par
The UWD is deduced by applying Theorem \ref{4.teor5.61}
to the families \eqref{4.hamilperzle}$_\ep^\alpha$ for a fixed
$\ep\in(0,\rho_0)$ and $\alpha$ varying in $\R$. Let
$N_\ep^\pm(\w,\alpha_0)$ be the corresponding principal functions.
Property \hyperlink{p5}{\bf p5} shows that
\begin{equation}\label{4.des1}
 N_{\ep_1}^+(\w,\alpha_0)\le N_{\ep_2}^+(\w,\alpha_0)\le
 N_{\ep_2}^-(\w,\alpha_0)\le N_{\ep_1}^-(\w,\alpha_0)
\end{equation}
for all $\w\in\W$ if $0<\ep_1\le\ep_2\le\rho_0$.
And, in order to prove the existence of ED, we must just prove that
$N_{\ep}^+(\w,\alpha_0)<N_\ep^-(\w,\alpha_0)$
for all $\w\in\W$ if $\ep\in(0,\rho_0)$: property \hyperlink{p3}{\bf p3}
ensures that in this case
$M_{\ep}^\pm(\w,\alpha_0)=N_{\ep}^\pm(\w,\alpha_0)$
and hence that $M_{\ep}^+(\w,\alpha_0)<M_{\ep}^-(\w,\alpha_0)$
if $\ep\in(0,\rho(\alpha_0))$. This will be done in
point \hyperlink{7}{\bf s7\,} below, after some preliminary work.
\par
Let us consider the new auxiliary families
\begin{equation}\label{4.hamilperzlE}
 \bz'=H_\mu^\lambda(\wt)\,\bz\,,\quad
 \text{where }\; H_\mu^\lambda(\w):=\left[\begin{array}{cc} H_1(\w)&
 H_3(\w)+\mu I_n\\
 H_2(\w)-\alpha\Delta(\w)&-H_1^T(\w)\end{array}\right]\bz\,,
\end{equation}
for $\mu\in\C$ and $\alpha\in\mI$,
which agree with \eqref{4.hamilperzle}$_\ep^\alpha$ if $\mu=\ep\in\R$.
If we take a real value of $\mu$, say $\ep_0$, in the interval $(0,\rho(\alpha_0))$,
then \eqref{4.hamilperzlE}$_{\ep_0}^{\alpha_0}$ has ED. Hence, according
to Theorem~\ref{3.teornoatki}(i), all the families
\eqref{4.hamilperzlE}$_{\mu}^{\alpha_0}$ have ED if $\Ima\mu>0$, and
the Weyl matrices $M^\pm_\mu(\w,\alpha_0)$ determine Herglotz
functions on $\C^+$ for each $\w\in\W$ fixed,
namely $\mu\mapsto M^\pm_\mu(\w,\alpha_0)$, with
$\pm\Ima M^\pm_\mu(\w,\alpha_0)\ge 0$.
The same reason justifies the existence of the Weyl functions
(with the same properties)
$M^\pm_\mu(\w,\alpha_m)$ for all $m\in\N$ if $\Ima\mu>0$.
Recall also that $(0,\rho_0)\subseteq(0,\rho(\alpha_m))$
for all $m\ge 1$. These facts allow will be used in the proof of the
following assertion:
\begin{itemize}
\item[\hypertarget{6}\text{\bf s6.}]
For all $\w\in\W$ and all
$\ep\in(0,\rho_0)$ there exist the limits $\wit F^\pm_{\ep}(\w,\alpha_0):=
\lim_{\beta\to 0^+} M^\pm_{\ep+i\beta}(\w,\alpha_0)$,
and they are real matrices. Moreover,
the matrix-valued functions $t\mapsto\wit F_\ep^\pm(\wt,\alpha_0)$
are two globally defined solutions of the Riccati
equation associated to \eqref{4.hamilperzlE}$_{\ep_0}^\alpha$.
In particular, $N^+_\ep(\w,\alpha_0)\le\wit F_\ep^\pm(\w,\alpha_0)
\le N_\ep^-(\w,\alpha_0)$
\end{itemize}
Let us sketch this proof in the case of $M^+$: it adapts the arguments leading
us to the proof of properties \hyperlink{3}{\bf s3\,} and
\hyperlink{4}{\bf s4\,} in Theorem~\ref{4.teorizq}, where all the details
are provided. We
fix $\w\in\W$ and $\ep\in(0,\ep_0)$, and represent
\[
 M_{\ep+i\beta}^+(\w,\alpha_m)=L_m+K_m\,(\ep+i\beta)+
 \int_\R\left(\frac{1}{t-\alpha-i\beta}-\frac{t}{t^2+1}\right)dP_m(t)
\]
for $m\ge 1$ and
\[
 M_{\ep+i\beta}^+(\w,\alpha_0)=L_0+K_0\,(\ep+i\beta)+
 \int_\R\left(\frac{1}{t-\ep-i\beta}-\frac{t}{t^2+1}\right)dP_0(t)
\]
Let us take $m\ge 1$. Then
\[
 \frac{1}{2}\,(P_m\{\ep_1\}
 +P_m\{\ep_2\})+\int_{(\ep_1,\ep_2)}dP_m(t)=
 \frac{1}{\pi}\,\int_{\ep_1}^{\ep_2}
 \Ima M_\ep^+(\wD,\alpha)\,d\ep=0_n
\]
whenever $0<\ep_1<\ep_2<\rho_0$. This ensures that
$\int_{(0,\rho_0)} dP_m(t)=0_n$ if $m\ge 1$,
which together with the property
$\lim_{m\to\infty}M_{\ep+i\beta}^+(\w,\alpha_m)=M_{\ep+i\beta}^+(\w,\alpha_0)$
uniformly on the compact sets of $\C^+$ allows us to deduce that
$(dP_m)$ converges to $dP_0$ in the weak$^*$ sense, and hence that
$\int_{(0,\rho_0)}dP_0(t)=0_n$.
\par
Let us take $\delta\in(0,\rho_0/2)$ and
$\mC_\delta:=
\{\ep+i\beta\in\C\,|\;\ep\in[\delta,\rho_0-\delta]\text{ and }
\beta\in(0,1]\}$.
Note that there exist constants $c_1>0$ and $c_2>c_1$ such that
\[
 c_1\le\Frac{|t-\ep-i\beta|^2}{t^2+1}\le c_2
 \quad\text{for $\;t\notin(0,\rho_0)\;$ and
 $\;\ep+i\beta\in\mC_\delta$}\,,
\]
which together wit the previous property leads to
\[
 0_n\le
 \Ima M^+_{\ep+i\beta}(\w,\alpha_0)\le
 \beta\left(K_0+\frac{1}{c_1}\,\Ima M_{\ep+i}^+(\w,\alpha_0)
 \right)
\]
if $\ep\in[\delta,\rho-\delta]$ and $\beta\in(0,1]$. In particular,
$\lim_{\beta\to 0^+}\Ima M_{\ep+i\beta}(\w,\alpha_0)=0_n$
uniformly on the compact subsets of $(0,\rho_0)$.
The Schwarz reflection principle allows us to ensure that
there exists
$\wit F^+_{\ep}(\w,\alpha_0):=
\lim_{\beta\to 0^+} M^+_{\ep+i\beta}(\w,\alpha_0)$,
and it is a real matrix. The arguments used at the end of the
proof of \hyperlink{4}{\bf s4} and to obtain \eqref{4.c1}
complete the proof of \hyperlink{6}{\bf s6} in the case of
$M^+$. And the case of $M^- $ is proved in the same way.
\par
The information provided by \hyperlink{6}{\bf s6\,} will allow us to
adapt the proof of
property \hyperlink{5}{\bf s5\,} in Theorem~\ref{4.teorizq} in order to
conclude that
\begin{itemize}
\item[\hypertarget{7}\text{\bf s7.}]
 $N_{\ep}^+(\w,\alpha_0)<N_\ep^-(\w,\alpha_0)$ for all $\w\in\W$ if $\ep\in(0,\rho_0)$,
\end{itemize}
which, as said before, shows the existence of ED
for $\ep\in(0,\rho_0)$.
\par
We proceed by contradiction, assuming the existence of
$\wit\w\in\W$, $\wit\ep\in(0,\rho_0)$, and
$\bv\in\R^n-\{\bcero\}$
such that $\bv^T\big(N_{\wit\ep}^-(\wit\w,\alpha_0)-
N^+_{\wit\ep}(\wit\w,\alpha_0)\big)\,\bv=0$. This fact and \eqref{4.des1}
ensure that
\begin{equation}\label{4.des2}
 \bv^T\big(N_{\ep}^-(\wit\w,\alpha_0)-N^+_{\ep}(\wit\w,\alpha_0)
 \big)\,\bv=0 \quad\text{if $\ep\in(\wit\ep,\rho_0)$}\,.
\end{equation}
Now we define $h\colon\C^+\to\C$ by
\[
 h(\ep+i\beta):=\bv^T\big(M_{\ep+i\beta}^-(\w,\alpha_0)-
 M^+_{\ep+i\beta}(\w,\alpha_0)\big)\,,
\]
which is holomorphic. In addition,
properties \hyperlink{6}{\bf s6\,} and \hyperlink{p4}{\bf p4} ensure that
\[
 N_\ep^+(\wit\w,\alpha_0)\le
 \wit F_\ep^\pm(\wit\w,\alpha_0)\le N_\ep^-(\wit\w,\alpha_0)
 \quad\text{if $\ep\in(\wit\ep,\rho_0]$}\,.
\]
Therefore, we deduce from \eqref{4.des2} that
\[
 \lim_{\beta\to 0^+}h(\ep+i\beta)=
 \bv^T\big(\wit F^-_{\ep}(\wit\w,\alpha_0)-
 \wit F^+_{\ep}(\wit\w,\alpha_0)\big)\,\bv=0
 \quad\text{for $\ep\in(\wit\ep,\rho_0)$}\,.
\]
Hence, there exists a holomorphic extension $\wit h$ of $h$ to
the set $(\C-\R)\cup(\wit\ep,\rho_0)$ such that $\wit h(\wit\w,\ep)=0$
for $\ep\in\R-[\wit\ep,\rho_0]$.
The principle of isolated zeroes ensures then that
\[
 \wit h(\ep+i\beta)=\bv^T\big(M^-_{\ep+i\beta}(\wit\w,\alpha_0)
 -M^+_{\ep+i\beta}(\wit\w,\alpha_0)\big)\,\bv=0
 \quad\text{if $\ep\in\R$ and $\beta>0$}\,.
\]
Taking now limits at a point $\ep\in(0,\rho(\alpha_0))$ yields
$\bv^T\big(M^-_{\ep}(\wit\w,\alpha_0)
 -M^+_{\ep}(\wit\w,\alpha_0)\big)\,\bv=0$. But
this is impossible, since, as seen before,
$M^-_{\ep}(\wit\w,\alpha_0)-M^+_{\ep}(\wit\w,\alpha_0)>0$.
This is the sought-for contradiction:
the proofs of \hyperlink{7}{\bf s7\,} and the first assertion in (ii) are complete.
The second assertion in (ii) is proved by property \hyperlink{p3}{\bf p3}.
\smallskip\par
(iii) The first assertion in (iii) has been checked at the beginning
of the first step in the proof of Theorem~\ref{4.teorizq}. In order
to check that $\rho$ is injective at the (perhaps nonexistent) interval
at which it takes real values, we take $\wit\alpha\in\mI$ such that
$\wit\rho:=\rho(\wit\alpha)<\infty$. Let us consider the families
\eqref{4.hamilperzle}$_{\wit\rho}^\alpha$ for $\alpha$
varying in $\R$. Since \eqref{4.hamilperzle}$_{\wit\rho}^{\wit\alpha}$
is UWD, Theorem \ref{4.teor5.61} ensures that
\eqref{4.hamilperzle}$_{\wit\rho}^{\wit\alpha}$ has ED and satisfies
the NC for $\alpha<\wit\alpha$. And this ensures that
$\rho(\alpha)>\wit\rho$ if $\alpha<\wit\alpha$, which proves the injectivity
of the map.
\smallskip\par
(iv) Assume for contradiction that $\alpha^*<\infty$ and that the family
\eqref{4.hamilperzl}$^{\alpha^*}$ has ED. Then there exist
$\ep_0>0$ and $\alpha_0<\alpha^*$ such that the families
\eqref{4.hamilperzle}$_\ep^{\alpha}$ have ED if
$\ep\in[0,\ep_0]$ and $\alpha\in[\alpha_0,\alpha^*]$:
see Theorem~\ref{2.teorSSPT}.
Consequently, since the map $\rho$ is nonincreasing, the
families \eqref{4.hamilperzle}$_\ep^{\alpha}$ are UWD for
$\ep\in(0,\ep_0]$ and $\alpha\in[\alpha_0,\alpha^*]$; and, in addition,
for these values of $\ep$ and $\alpha$ there exists
$M^+_\ep(\w,\alpha)=N^+_\ep(\w,\alpha)$ (see property \hyperlink{p3}{\bf p3}).
Property \hyperlink{p5}{\bf p5} ensures that the two-parametric family
$M^+_\ep(\w,\alpha)$ increases with $\alpha$ and decreases with $\ep$. Let
us take a sequence $(\ep_m)\downarrow 0$ and a point
$\alpha_1\in(\alpha_0,\alpha^*)$.
Then
$M^+(\w,\alpha_1)=\lim_{m\to \infty}M_{\ep_m}^+(\w,\alpha_1)$, so that
\[
 M^+(\w,\alpha_1)\le M^+_{\ep_m}(\w,\alpha^1)\le M^+_{\ep_m}(\w,\alpha^*)
\]
for all $m\ge 1$. In particular, the sequence of matrices
$(M_{\ep_m}(\w,\alpha^*))$, which decreases, is bounded from below.
Therefore, for any
$\w\in\W$, there exists a suitable convergent subsequence
(use the polarization formulas). The continuous variation of
the Lagrange planes associated to the ED ensures that
this limit is necessarily $M^+(\w,\alpha^*)$.
Consequently, the family \eqref{4.hamilperzl}$^{\alpha^*}$
satisfies NC, which ensures that $\alpha^*\in\mI$. This is the
sought-for contradiction. The proof is complete.
\hfill{\qed}
\medskip\par
Note that, under Hypothesis~\ref{4.hipos}, the set $\mI$
defined by \eqref{4.defI} can either
be upper bounded or agree with the whole real line. And, if it
is bounded, then the value of $\rho_0:=\lim_{\alpha\to(\sup\mI)^-}\rho(\alpha)$
can be 0, a positive real value, or $\infty$. We check these assertions
by means of simple autonomous examples:
\begin{list}{}{\leftmargin 12pt}
\item[-] In the case of $\bz'=\lsm -1&0\\-\lambda&1\rsm\,\bz$,
$\mI=\R$ and $\rho_0=0$. More precisely, $M^+(\lambda)=\lambda/2$ (and
$l^-(\lambda)=\{\lsm 0\\x\rsm\,|\;x\in \R\}$, so that $M^-$
does not exist) for all $\lambda\in\R$; and
with $\rho(\lambda)=\infty$ for $\lambda\le 0$ and $\rho(\lambda)=1/\lambda$
for $\lambda>0$.
\item[-]
In the case $\bz'=\lsm -1&1\\\;\,-\lambda&\;\,1\rsm\,\bz$,
$\mI=(-\infty,1)$ and $\rho_0=0$. More precisely,
$\rho(\lambda)=\infty$ for $\lambda\le 0$
and $\rho(\lambda)=-1+1/\lambda$ for $\lambda\in(0,1)$.
\item[-]
In the case $\bz'=\lsm 0&1\\\;\,1-\lambda&\;\,0\rsm\,\bz$,
$\mI=(-\infty,1)$ and $\rho_0=\infty$, since
$\rho(\lambda)=\infty$ for $\lambda< 1$.
\item[-]
Finally, combining the first and third examples, we
obtain the 4-dimensional system
\[
\bz'=\left[\begin{array}{rrrr}-1&0&0&0\\0&0&0&1\\-\lambda&0&1&0\\
0&1-\lambda&0&0\end{array}\right]\bz\,,
\]
for which $\mI=(-\infty,1)$ and $\rho_0=\lim_{\alpha\to 1^-}1/\alpha=1$.
\end{list}
\par
However, if $\sup\mI=\infty$, then $\lim_{\alpha\to(\sup\mI)^-}\rho(\alpha)=0$
(as it happens in the first example). This assertion follows from
Theorem~\ref{4.teorlibro}(ii) and Theorem \ref{4.teor5.61}
applied to the families
\eqref{4.hamilperzle}$^{\ep_0}_{\alpha}$ for a fixed $\ep_0>0$ and
$\alpha$ varying in $\R$. The same results, combined with
the robustness of the properties of ED and global existence of $M^+$ (see
Theorem~\ref{2.teorSSPT}) ensure that, if $\mI$ is bounded
and $\rho_0>0$, then the families
\eqref{4.hamilperzle}$^\ep_{\rho_0}$ corresponding to $\ep\in(0,\rho_0)$
are UWD but they do not have ED.
\begin{nota}\label{2.minimal}
Theorem \ref{4.teormejor} can be easily extended to the case that the base flow $(\W,\sigma)$ is distal. In this case,
$\W$ decomposes in the disjoint union of a family
of (distal) minimal sets (see Ellis~\cite{elli}).
Therefore, we can apply Theorem \ref{4.teormejor}
over each minimal component. Since the ED of a linear family is
equivalent to the ED of each one of its systems (see Remark~\ref{2.notasED}(b)),
and obviously the same happens with the existence of $M^+$, we
conclude that the conclusion of the theorem also hold over the whole of $\W$.
\end{nota}

\end{document}